\documentclass[11pt,reqno,final]{amsart}
\usepackage[T1]{fontenc}
\usepackage{amsmath,amsthm}
\usepackage{amsfonts}  
\usepackage{amssymb}
\usepackage{hyperref}
\usepackage{graphicx} 
\usepackage{bm}
\usepackage{xcolor}
\usepackage{nicefrac}

\usepackage{pgfplots}
\pgfplotsset{compat=1.16}
 
\usepackage{tikz}\usetikzlibrary{patterns}
\usepackage[paperwidth=210mm,paperheight=255mm,top=20mm,bottom=25mm,inner=25mm,outer=25mm,marginparsep=7mm,marginparwidth=48mm,headsep=4mm,headheight=6mm]{geometry}
\usepackage{stmaryrd}

\usepackage{bbold}

\usefont{T1}{cmr}{m}{n}
\usepackage[color,notref,notcite]{showkeys}   
\definecolor{labelkey}{rgb}{0.6,0,1}

\usepackage[normalem]{ulem}
\newcounter{corr}
\definecolor{violet}{rgb}{0.580,0.,0.827}
\newcommand{\corr}[3]{\typeout{Warning : a correction remains in page
\thepage}
				\stepcounter{corr}        
				{\color{blue}\ifmmode\text{\,\sout{\ensuremath{#1}}\,}\else\sout{#1}\fi}
       {\color{red}#2}
       {\color{violet} #3}}
\numberwithin{equation}{section}
\usepackage{constants}
\newconstantfamily{ctrcst}{symbol=C,reset={bibunit}}
\def\ctel#1{\ensuremath{\Cl[ctrcst]{#1}}}
\def\cter#1{\ensuremath{\Cr{#1}}}

 \newtheorem{thm}{Theorem}[section]
 \newtheorem{lem}[thm]{Lemma}
 \newtheorem{remark}[thm]{Remark}
 \newtheorem{definition}[thm]{Definition}

\newcommand{\mesh}{\mathcal T}

 \def\gnc{\mathcal{G}}

 \newcommand\nti{n \to +\infty}
\renewcommand{\d}{{\mathcal \partial}}

\newcommand{\edges}{\mathcal{E}}
\newcommand{\edgesint}{\mathcal{E}_{\mathrm{int}}}
\newcommand{\edgesext}{\mathcal{E}_{\mathrm{ext}}}

\newcommand{\dx}{\, \mathrm{d}x}

\newcommand{\dsp}{\displaystyle} 
\newcommand{\mathbi}[1]{{\boldsymbol #1}}

\newcommand{\bfn}{\mathbi{n}}

\newcommand{\bv}{\mathbi{v}}

\newcommand{\bF}{\mathbi{F}}

\newcommand{\bvarphi}{\mathbi{\varphi}}

\renewcommand{\d}{{\rm d}}
\newcommand{\dfrontiere}{\d{\rm s}}
\newcommand{\dr}{\partial}
\renewcommand{\div}{{\mathop{\rm div}}}

\newcommand{\grad}{\nabla}

\newcommand{\hdiv}{\mathbi{H}_{\rm div}}

\renewcommand{\O}{\Omega}

\renewcommand{\phi}{\varphi}

\newcommand{\R}{\mathbb R}

\newcommand{\Xmesh}{X_\mesh}

\newcommand{\ba}{\begin{array}{llll}   }
\newcommand{\ea}{\end{array}}

\newcommand{\be}{\begin{equation}}
\newcommand{\ee}{\end{equation}}

\newcommand{\act}[4]{\langle#1,#2\rangle_{#3,#4}}

\author{R. Eymard}
\address{Robert Eymard: Universit\'e Gustave Eiffel, LAMA, (UMR 8050), UPEM, UPEC, CNRS, F-77454, Marne-la-Vallée (France)}
\email{robert.eymard@univ-eiffel.fr}

\author{T. Gallou\"{e}t}
\address{Thierry Gallou\"et: I2M UMR 7373, Aix-Marseille Universit\'e, CNRS, Ecole Centrale de Marseille,
F-13453 Marseille, France}
\email{thierry.gallouet@univ-amu.fr}

\author{R. Herbin}
\address{Rapha\`ele Herbin: I2M UMR 7373, Aix-Marseille Universit\'e, CNRS, Ecole Centrale de Marseille,
F-13453 Marseille, France}
\email{raphaele.herbin@univ-amu.fr}

\title[Optimal error bounds for the TPFA scheme]{Optimal error bounds for the \\Two Point Flux Approximation finite volume scheme}

\begin{document}	

\begin{abstract} 
We consider a finite volume scheme with two-point flux approximation (TPFA) to approximate a Laplace problem when the solution exhibits no more regularity than belonging to $H^1_0(\Omega)$. 
We establish in this case some error bounds for both the solution and the approximation of the gradient component orthogonal to the mesh faces. 
This estimate is optimal, in the sense that the approximation error has the same order as that of the sum of the interpolation error and a conformity error. 
A numerical example illustrates the error estimate in the context of a solution with minimal regularity. This result is extended to evolution problems discretized via the implicit Euler scheme in an appendix. 
\end{abstract}

\maketitle

 \tableofcontents

\section{Introduction}

Finite volume methods are standardly used for the approximation of elliptic and parabolic problems in several physical or engineering frameworks such as fluid mechanics, reservoir simulation, heat and mass transfer, biology, biomedical research\ldots
Here we are more specifically  interested  in the two-point flow approximation (TPFA) finite volume scheme for the approximation of the Laplace operator; this is a very popular scheme that has been used since the 60's in the various above mentioned applications, see e.g. \cite{aziz},\cite{COUDIERE20131581} to cite only a few of these. 

Important features of the TPFA scheme are $L^\infty$ stability and monotony \cite{book}. 
Thanks to these properties, stable and convergent numerical schemes based on TPFA were designed for nonlinear problems \cite{can2020behavior} or linear problems with singular source terms \cite{dro2003noncoer}.

A mathematical proof of convergence of the TPFA scheme for the Laplace equation was first given for rectangular grids in  \cite{FORSYTH1988377}, and the  first comprehensive mathematical analysis of this scheme fore more general grids (the so called admissible grids) dates back to 2000  \cite{book}.
Therein, under suitable assumptions on the physical domain and the right hand side, it is proven that if the exact solution of the Poisson equation posed on a open polytopal bounded subset of $\mathbb{R}^d$ with $d=2$ or $3$ is assumed to be in $H^2$, then an error estimate of order 1 can be obtained.
Superconvergence has been observed and proven for a class of triangular meshes, see \cite{droniou-nataraj} and Remark \ref{rem:superconvergence} below. 

Our aim in the present paper is to give an error estimate in the case where the right-hand side belongs to the dual space of the energy space associated to the equation, namely $H^{-1}$ and without any supplementary regularity assumption on the exact solution other than the natural regularity given by the problem, namely $H^1_0$.
The case of a right-hand side in $H^{-1}$ had already been studied in \cite{Droniou_Gallouet_2002}, where the convergence of the scheme is proven for a non coercive elliptic operator. 
In the present paper, we study the same scheme for the Laplace equation, and prove an optimal bound on the error;  more precisely we show that the discretization error is bounded by below and by above, in each case up to a constant, by the sum of two errors, namely the interpolation error and the conformity error, whose precise definitions are given in the sequel. 
We also obtain a bound of the error between an approximate gradient of the solution and the gradient of the exact solution. 
As far as we know, both results are original.

In order to introduce these results, let us first give some definitions. 
Consider $f\in L^2(\Omega)$ and  $\Omega$  an open polytopal bounded subset of $\mathbb{R}^d$ with $d=2$ or $3$; denote by  $\overline{u}\in H^1_0(\Omega)$ the solution of the problem
\begin{equation}
 \forall v\in H^1_0(\Omega),\ \int_\Omega \nabla \overline{u}(x)\cdot \nabla v(x){\rm d} x = \int_\Omega f(x) v(x){\rm d} x.
 \label{eq:laplace_intro}
\end{equation}
As above mentioned, if $\overline{u}\in H^2(\Omega)$, there exists $C>0$ only depending on a regularity factor of the finite volume mesh $\mesh$ such that the solution $u$ of the finite volume scheme satisfies
\begin{equation}\label{eq:introerrest}
  \Vert u - \overline{u}\Vert_{L^2(\Omega)} \le C h_{\mesh} \Vert \overline{u}\Vert_{H^2(\Omega)},
\end{equation}
where $h_{\mesh}$ is the size of the mesh. 

Now when seeking an approximate solution $u_h\in V_h$ of the same Problem \eqref{eq:laplace_intro} by a conforming finite element method, which means that
\begin{equation}
 \forall v_h\in V_h,\ \int_\Omega \nabla u_h(x)\cdot \nabla v_h(x){\rm d} x = \int_\Omega f(x) v_h(x){\rm d} x,
 \label{eq:laplaceconf}
\end{equation}
we get that
\[
  \Vert \overline{u} - u_h\Vert_{H^1_0(\Omega)} = \inf_{v_h\in V_h} \Vert \overline{u} - v_h\Vert_{H^1_0(\Omega)}, 
\]
where:
\begin{itemize}
\item $\Vert \cdot \Vert_{H^1_0(\Omega)}$ is the energy norm associated to Problem \eqref{eq:laplace_intro}, that is to say $\Vert v \Vert_{H^1_0(\Omega)} = \int_\Omega |\nabla v(x)|^2\dx$,
 \item $V_h\subset  H^1_0(\Omega)$ is the (finite dimensional) finite element space.
\end{itemize}
This result, which is a particular case of Céa's lemma for the problem at hand, states that the energy norm of the discretization error is equal to the energy norm of the interpolation error~; it holds without any further regularity assumption than $\overline{u}\in H^1_0(\Omega)$.
It is optimal in the sense that the order of convergence is given by the order of the interpolation error $ \inf_{v\in V_h} \Vert \overline{u} - v\Vert_{H^1_0(\Omega)}$. 

Such an optimal error bound is extended in \cite{gdm} to some nonconforming schemes that fall in the framework of the Gradient Discretization Method (GDM). 
The GDM was invented and analysed some time ago in order to construct a framework for a number of schemes that have been devised in the past twenty years in order to deal with anisotropic diffusion problems and/or distorted meshes. 
Let us briefly recall the basics of the method. 
Let $X_\mesh$ be a finite dimensional vector space of the degrees of freedom generated by a mesh $\mesh$,  $\Pi_{\mesh}  v$ denote a  function that is reconstructed from any $v\in X_\mesh$ and  defined a.e. in $\Omega$, and $\nabla_{\mesh} v$ be an approximate gradient, also reconstructed from $v$.  Then Problem \eqref{eq:laplace_intro} is approximated by $u\in X_\mesh$ such that
\begin{equation}
 \forall v\in X_\mesh,\ \int_\Omega \nabla_{\mesh} u(x)\cdot \nabla_{\mesh} v(x){\rm d} x = \int_\Omega f(x) \Pi_{\mesh} v(x){\rm d} x.
 \label{eq:laplacegdm}
\end{equation}
Note that for non conforming schemes, $\Pi_{\mesh} (X_\mesh) \not\subset H^1_0(\Omega)$, the approximate gradient $\nabla_{\mesh} v$ cannot be defined as the continuous gradient of $\Pi_{\mesh}  v$. 
This difficulty is overcome by defining the conformity error of the scheme by
\begin{equation}
 \zeta_\mesh(\nabla \overline{u}) = \max_{v\in X_\mesh\setminus \{0\} } \frac {\langle \nabla \overline{u},\nabla_{\mesh}  v\rangle_{L^2} +  \langle {\rm div}(\nabla \overline{u}),\Pi_{\mesh}  v\rangle_{L^2}} { \Vert\nabla_{\mesh}  v\Vert_{L^2} }.
 \label{eqdef:conformity}
\end{equation}
An optimal error bound (see \cite[Theorem 2.28]{gdm}) may then be obtained in the spirit of second Strang's lemma \cite{str-72-var}; introducing 
\begin{equation}
\delta_\mesh(\overline{u},v)^2 = \Vert \overline{u} - \Pi_{\mesh}v\Vert_{L^2}^2 + \Vert \nabla\overline{u} - \nabla_{\mesh}v\Vert_{L^2}^2,
\label{eqdef:approximation}
\end{equation}
this bound states that the error $\delta_\mesh(\overline{u},u)$ committed on the solution and its gradient satisfies
\begin{equation}
 \frac 1 2  \Big(\zeta_\mesh(\nabla \overline{u}) + \inf_{v\in X_\mesh}\delta_\mesh(\overline{u},v)\Big) \le \delta_\mesh(\overline{u},u) \le C\Big(\zeta_\mesh(\nabla \overline{u}) + \inf_{v\in X_\mesh} \delta_\mesh(\overline{u},v)\Big).
 \label{optimal_bound}
\end{equation}
without requiring any more regularity than $\overline{u}\in H^1_0(\Omega)$ on the exact solution.
This bound is said to be optimal in the sense that the order of the approximation error $\delta_\mesh(\overline{u},u)$ is the same as that of the sum of the interpolation error   $\inf_{v\in X_\mesh} \delta_\mesh(\overline{u},v)$ and the conformity error $\zeta_\mesh(\nabla \overline{u})$. 

\medskip

In some particular situations (rectangular or acute triangular cells), the TPFA FV scheme can be shown to be a GDM (see \cite[Lemma 13.20]{gdm}), and therefore  the optimal bound \eqref{optimal_bound} holds. 
However, in the case of general admissible meshes of Definition \ref{def:meshdirichlet} below, this is no more the case; for instance the TPFA FV scheme on  Vorono\"{\i} meshes \cite{wias} cannot be seen as a GDM.

\medskip

The aim of the present paper is to prove an error estimate for the TPFA finite volume scheme which
\begin{itemize}
 \item[-]  is optimal in the sense of the bound \eqref{optimal_bound},
 \item[-] still holds for meshes which only satisfy the assumptions of  Definition \ref{def:meshdirichlet},
 \item[-]  applies to general right-hand sides in $H^{-1}(\Omega)$ so that the exact solution belongs only to $H^1_0(\Omega)$,
 \item[-] and is formulated through a stronger norm than the   $L^2(\Omega)$ norm.
\end{itemize}
To this purpose, we first propose to use an approximate gradient in the right-hand side (the so-called ``inflated approximate gradient'') of the scheme in order to deal with for right-hand sides in $H^{-1}(\Omega)$. 
Unfortunately this approximate gradient cannot be used for a GDM scheme, due to the fact that it can only weakly converge. 
A variational formulation of the TPFA FV scheme is then obtained by considering only the normal component of the approximate gradient.
We may then follow the line of thought of \cite[Theorem 2.28]{gdm} and obtain an optimal error bound involving an interpolation error between the approximation of the normal gradient and the normal component of the gradient of the continuous solution. 
This result is, to our knowledge, original for at least these two reasons: it does not require more regularity on the exact solution than the natural regularity obtained from the continuous problem, and it leads to  an error estimate between a consistent reconstruction of an approximate gradient (necessarily different from the inflated approximate gradient) and the full gradient of the continuous solution. 

\medskip

We then apply these two results in the case where the continuous solution is in $H^2$, extending the results of \cite{book} to the convergence of the consistent approximate gradient.

\medskip

A numerical example in the case of a solution with minimal regularity illustrates in Section \ref{sec:num} the optimality of the error estimate; indeed the considered solution is unbounded and does not belong to any $W^{1,p}_0(\Omega)$, for $p>2$.

A short conclusion follows.

\medskip

In Appendix \ref{sec:heat}, following the methods of proof provided by \cite{RDV2}  and \cite{droniou:hal-04183945}, we extend the optimal error bound obtained for the Laplace problem to the transient heat equation, using an implicit Euler time discretization. 

\section{The finite volume scheme for the Laplace problem}\label{sec:steady}

Let $\Omega\subset\mathbb{R}^d$ be an open polytopal domain (with $d=2$ or $d=3$), with boundary $\partial\Omega$, and  $L \in H^{-1}(\Omega)$. 
We consider the Dirichlet problem 
\begin{align}
 & - \Delta \overline{u} = L \label{eq:laplace}\\
 & \overline{u} = 0 \mbox{ on }\partial \Omega.\label{eq:cldiric}
\end{align}
It is wellknown that there exists a unique solution to this problem, which satisfies
\begin{align*}
 & \overline{u} \in H^1_0(\Omega)\\
 & \int_\Omega  \nabla \overline{u} \cdot \nabla \varphi \dx = \act{L}{\varphi}{H^{-1}(\Omega)}{H^1_0(\Omega)}, \; \forall \varphi \in H^1_0(\Omega).
\end{align*}
It is also wellknown that any linear form $L \in H^{-1}(\Omega)$ may be decomposed as $L = f + \mathrm{div}\bF$ with $f\in L^2(\Omega)$ and $\bF\in L^2(\Omega)^d$.
Since $\act{\mathrm{div}\bF}{\varphi}{H^{-1}(\Omega)}{H^1_0(\Omega)} = - \int_\Omega \bF \cdot \nabla \varphi \dx,$   the above weak formulation may be recast as
\begin{align}\label{eq:pbellcont}
 &\overline{u}\in H^1_0(\Omega),\ \forall \varphi\in H^1_0(\Omega),\\ &\int_\Omega \nabla\overline{u}(x)\cdot\nabla\varphi(x) {\rm d}x = \int_\Omega (f(x)\varphi(x) - \bF(x)\cdot\nabla\varphi(x)) {\rm d}x.
\end{align}
Observe that 
\begin{equation} \label{dans:hdiv}
 \nabla\overline{u}+\bF\in \hdiv(\Omega)\mbox{ with }-{\rm div}(\nabla\overline{u}+\bF) = f\mbox{ a.e in }\Omega.
\end{equation}
 
 The remaining part of this section is dedicated to the definition of a finite volume for the approximation of this problem.

Let $\mesh$ be an admissible mesh of $\Omega$ in the following sense, close to that given in \cite[Definition 3.1]{book}.
\begin{definition}[Admissible meshes]\label{def:meshdirichlet}  
An admissible finite volume mesh of $\Omega$, denoted by $\mesh$, is given by a finite family of ``control volumes'', which are disjoint open polytopal convex subsets of $\Omega$, a finite family of disjoint subsets of $\overline{\Omega}$ contained in  hyperplanes of $\R^d$, denoted by  $\mathcal{E}$ (these are the edges (two-dimensional) or sides  (three-dimensional) of the control volumes), with strictly positive ($d-1$)-dimensional measure, and a family of points of $\Omega$ denoted by  $\mathcal{P}$ satisfying the  following properties (in fact, we shall denote, somewhat  incorrectly, by $\mesh$ the family of control volumes):

\begin{enumerate}
 
\item[(i)] The closure of the union of all the control volumes is
$\overline{\Omega}$;  

\item[(ii)]  For  any $K\in \mesh$, there exists  a subset $\mathcal{E}_K$ of $\mathcal{E}$
such that $\dr K  = \overline{K}\setminus K = 
\dsp{\cup_{\sigma \in \mathcal{E}_K}}
\overline{\sigma} $. Furthermore, $\mathcal{E}=\dsp{\cup_{K \in \mesh}\mathcal{E}_K}$.

\item[(iii)] The family  $\mathcal{P} = (x_K)_{K \in \mesh}$ is  such that
$x_K \in K$ for all $K \in \mesh$.

\item[(iv)] The set $\mathcal{E}$ is partitioned into $\mathcal{E} = \mathcal{E}_{\rm ext}\cup\edgesint$. For any $\sigma\in\mathcal{E}$:
\begin{itemize}
 \item[-] either $\sigma\in\mathcal{E}_{\rm ext}$; then there exists exactly one $K\in \mesh$  such that $\sigma\in \mathcal{E}_K$, $\sigma \subset \dr \O$ and the straight line $\mathcal{D}_{K,\sigma}$ going through $x_K$  and orthogonal to $\sigma$ is such that $\mathcal{D}_{K,\sigma} \cap \sigma \neq\emptyset$;
\item[-] or $\sigma\in\edgesint$; then there exist exactly two elements of $\mesh$ denoted $K$ and $L$ such that $\sigma\in \mathcal{E}_K\cap \mathcal{E}_L$, $\dr K\cap \dr L = \overline{\sigma}$ and the straight line $\mathcal{D}_{K,L}$ going through $x_K$ and $x_L$ is orthogonal to $\sigma$ ($\mathcal{D}_{K,L}\cap\overline{\sigma} =\emptyset$ is not excluded); in this case $\sigma$ is denoted by $K|L$. 
\end{itemize}
\end{enumerate}

Figure \ref{fig:mesh} shows an example of mesh with a few notations.
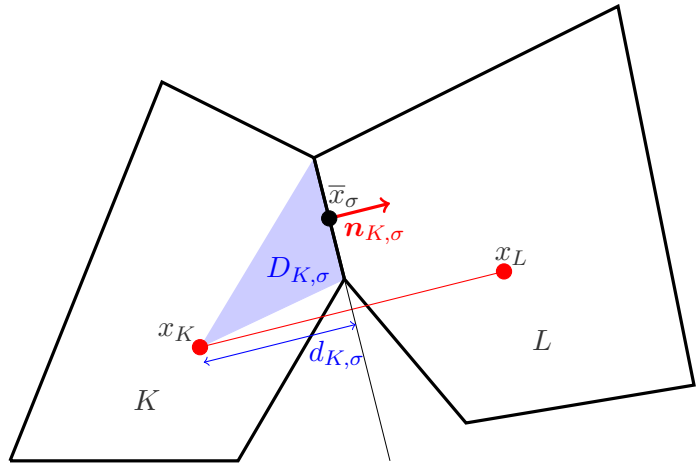
\begin{figure}[!ht]
  \begin{center}
  \begin{tikzpicture}[scale=1]

   \draw[-, blue!20, pattern=north west lines, pattern color=blue!30, opacity=0.0] (3.5,2.5)--(5.4,3.4)--(5,5); 
   \draw[-, very thick, black](1,1)--(4,1)--(5.4,3.4)--(5,5)--(3,6)--(1,1);
   \draw[-, very thick, black](5,5)--(5.4,3.4)--(7,1.5)--(10,2)--(9,7)--(5,5);
   \draw[-, , red](3.5,2.5)--(7.5,3.5);
   \draw[-, very thin, black](5,5)--(6,1);
   \draw[<->, , blue](3.55,2.3)--(5.55,2.8);
   \draw[->,very thick, red](5.2,4.2)--(6,4.4);
   \path(5.3,2.4) node[blue!100]{$d_{K,\sigma}$};
   \path(5.8,4) node[red!100]{${\bm n}_{K,\sigma}$};

   \draw[fill,red](3.5,2.5) circle (0.1);
   \draw[fill,red](7.5,3.5) circle (0.1);
   \draw[fill,black](5.2,4.2) circle (0.1);
   
   \path(3.2,2.7) node[black!80]{$x_K$};
   \path(7.6,3.7) node[black!80]{$x_L$};
   \path(4.8,3.5) node[blue!100]{$D_{K,\sigma}$};
   \path(2.8,1.8) node[black!80]{$K$};
   \path(8,2.6) node[black!80]{$L$};
   \path(5.4,4.5) node[black!80]{$\overline{x}_\sigma$};
   
  \end{tikzpicture}
  \caption{Two neighbouring control volumes of an admissible mesh. }\label{fig:mesh}
  \end{center}
\end{figure}

The following notations are used
\begin{itemize}
 \item $|A|$ the $d$ (resp. $(d-1)$)-dimensional  measure of any subset $A$ of $\R^d$ (resp. $\R^{d-1})$.
 \item $\overline{x}_\sigma$: center of gravity of $\sigma \in \mathcal{E}$,
  \item $d_{K,\sigma}>0$:  orthogonal distance from $x_K$ to $\sigma$, for $K\in \mesh$ and $\sigma\in \mathcal{E}_K$,
  \item $D_{K,\sigma}$:  cone with vertex $x_K$ and basis $\sigma$, for $K\in \mesh$ and $\sigma\in \mathcal{E}_K$, so that $|D_{K,\sigma}| = \frac{|\sigma|d_{K,\sigma}} d,$ 
  \item ${\bm n}_{K,\sigma}$:  unit vector, normal to $\sigma$ and outward to $K$
  \item $h_{\mesh}= \sup\{{\rm diam}(K),$ $K\in \mesh\}$:   mesh size,
  \item $\theta_{\mesh}= \inf\{\frac {d_{K,\sigma}}{{\rm diam}(K)},$ $K\in \mesh,\ \sigma\in\mathcal{E}_K\}$: mesh regularity parameter,
  \item $\Xmesh$: set of all real families $u = \Big((u_K)_{K\in \mesh},(u_\sigma)_{\sigma\in \mathcal{E}}\Big)$ such that $u_\sigma = 0$ if $\sigma\in\mathcal{E}_{\rm ext}$.
\end{itemize}
With a slight abuse of notation, for any $u\in \Xmesh$, we also denote by $u$ the element of $L^2(\Omega)$ which is a.e. equal to the value $u_K$ in any  $K\in\mesh$.
%
%
%
%

\end{definition}
Let us observe that this definition leads to admissible meshes in the sense of \cite[Definition 3.1]{book}, but is slightly more restrictive since, for the sake of simplicity and as in \cite{Droniou_Gallouet_2002}, we do not allow $x_K\in \overline{\sigma}$ if $\sigma\in \mathcal{E}_K$.
There are ways to relax this assumption by eliminating the face unknowns in the scheme, but this leads to additional technical difficulties. 

As mentioned in the introduction, our analysis requires several discrete derivation operators which we now precisely define.

\begin{definition}[Discrete derivative and gradients]\label{def:gradients} ~

\begin{itemize}
 \item   $G_{\mesh}$ : normal discrete derivative 
\vspace{-.05cm}
\begin{align}
  G_{\mesh}: & \Xmesh\to L^2(\Omega) \nonumber\\
            & u \mapsto  G_{\mesh}u, \mbox{ with } G_{\mesh}u(x) = \frac {u_\sigma - u_K} {d_{K,\sigma}} \text{ for a.e. } x\in D_{K,\sigma}.
\label{eq:defgm}
\end{align}
\vspace{-.2cm}
\item   $\nabla_{\mesh}$ : inflated discrete gradient  
\vspace{-.05cm}\begin{align}
\nabla_{\mesh}: & \Xmesh\to (L^2(\Omega))^d \nonumber\\
            & u \mapsto  \nabla_{\mesh} u, \mbox{ with } \nabla_{\mesh} u(x) = d\frac {u_\sigma - u_K} {d_{K,\sigma}} {\bm n}_{K,\sigma} \text{ for a.e. } x\in D_{K,\sigma}.
\label{eq:defnablam}
\end{align}

\vspace{-.05cm}   
\item $\widehat\nabla_{\mesh}$: consistent discrete gradient  
\vspace{-.2cm}\begin{align}
\widehat  \nabla_{\mesh}: & \Xmesh\to (L^2(\Omega))^d \nonumber\\
            & u \mapsto  \widehat \nabla_{\mesh} u, \mbox{ with } \widehat \nabla_{\mesh} u(x) = d\frac 1 {|K|} \sum_{\sigma\in\mathcal{E}_K} |\sigma|(\overline{x}_\sigma - x_K)  \frac {u_\sigma - u_K} {d_{K,\sigma}}\text{ for a.e. } x\in D_{K,\sigma}.
\label{eq:defgradconsist}
\end{align}
\vspace{-.2cm}
 
\end{itemize}
\end{definition}

\begin{remark}[On the inflated and consistent gradients]
 The \emph{inflated discrete gradient}, first introduced in \cite[Definition 2]{ar4} only involves the normal discrete gradient, but with a factor $d$, hence the term {\it inflated}; it also appears, but somewhat hidden, in the weak formulation (2.6) of the FV scheme in \cite{Droniou_Gallouet_2002}.
 The consistent discrete gradient $\widehat\nabla_{\mesh} u$ was first introduced in \cite[Definition 2.3]{vf100} in one of the first attempts to generalize finite volume schemes to anisotropic diffusion problems. 
\end{remark}

Following \cite{Droniou_Gallouet_2002}, integrating \eqref{eq:laplace} on a cell $K$ and (formally) integrating by parts yields the following balance equation on each cell $K$:
\[
-  \sum_{\sigma\in \mathcal{E}_K} \int_\sigma \nabla u \cdot \bfn_{K,\sigma} \dfrontiere  = \int_K f \dx + \sum_{\sigma\in \mathcal{E}_K} \int_\sigma \bF \cdot \bfn_{K,\sigma} \dfrontiere .
\]
Introducing the discrete unknowns $(u_K)_{K \in \mesh}$ and $(u_\sigma)_{\sigma \in \edges}$, and following \cite{Droniou_Gallouet_2002}, we propose the following scheme:
\begin{subequations}\label{strong-fv-scheme}
\begin{align}
& \forall K \in \mesh,  \; -  \sum_{\sigma\in \mathcal{E}_K} |\sigma| \frac{u_\sigma - u_K}{d_{K,\sigma}} =   \int_K f \dx +  \sum_{\sigma\in \mathcal{E}_K} |\sigma| \frac 1{|D_{K,\sigma}|}  \int_{D_{K,\sigma}} \bF\cdot \bfn_{K,\sigma},\label{strong-bal}\\
&\forall \sigma = K|L \in \edgesint, \nonumber \\
&\qquad \frac{u_\sigma - u_K}{d_{K,\sigma}} + \frac 1{|D_{K,\sigma}|} \int_{D_{K,\sigma}} \!\!\!\!\!\!\bF \cdot \bfn_{K,\sigma}\dx = - \frac{u_\sigma - u_L}{d_{L,\sigma}} - \frac 1{|D_{L,\sigma}|}  \int_{D_{L,\sigma}}  \!\!\!\!\!\! \bF \cdot \bfn_{L,\sigma}  \dx.\label{strong-cons}\\
& \forall \sigma \in \edgesext, \; u_\sigma = 0. \label{strong-cl}
\end{align}
\end{subequations}
Equation \eqref{strong-bal} is the discretization of the local mass balance on the cell $K$, while \eqref{strong-cons} expresses the conservativity of the discrete fluxes. 
By \cite[Theorem 2.1]{Droniou_Gallouet_2002}, there exists a unique solution to the scheme \eqref{strong-fv-scheme}.
Moreover, the scheme admits a weak formulation \cite[Lemma 2.1]{Droniou_Gallouet_2002}.

\begin{lem}[Weak formulation of the scheme]
 The scheme \eqref{strong-fv-scheme} is equivalent to the following weak formulation: 
\begin{align}
&\mbox{Find } u \in \Xmesh \mbox{ such that }, \mbox{ for any }v = \left((v_K)_{K \in \mesh},(v_\sigma)_{\sigma \in \edges}  \right)\in \Xmesh,
 \nonumber \\
\label{eq:fvscheme}
& d \int_\Omega G_{\mesh}u(x)G_{\mesh}v(x) {\rm d}x = \int_\Omega \left(f(x)v(x) - \bF(x)\cdot \nabla_{\mesh}v(x)\right) {\rm d}x, 
\end{align}
where $G_\mesh v$ and $ \nabla_{\mesh}v$ are the discrete derivative and gradient given in Definition \ref{def:gradients}.
\end{lem}
\begin{proof}
Owing to the definitions \eqref{eq:defgm} and \eqref{eq:defnablam} of  $G_\mesh v$ and $\nabla_{\mesh}v$ and noting that $|D_{K,\sigma}| = \frac {|\sigma|\, d_{K,\sigma}} {d}$,  the scheme \eqref{eq:fvscheme} also reads
\begin{equation}
 \sum_{K\in\mesh}  \sum_{\sigma\in\mathcal{E}_K} \frac {|\sigma|}{d_{K,\sigma}} (u_\sigma - u_K)(v_\sigma - v_K) =  \sum_{K\in\mesh}\int_K f v_K \dx-   \sum_{K\in\mesh}\sum_{\sigma\in \mathcal{E}_K} \int_{D_{K,\sigma}}\bF   d \frac {v_\sigma - v_K} {d_{K,\sigma}}   \cdot  \bfn_{K,\sigma}.
 \label{varfor-scheme}
\end{equation}
The proof that if $u$ is a solution to \eqref{strong-fv-scheme} then $u$ satisfies \eqref{varfor-scheme}  may be found in \cite[Lemma 2.1]{Droniou_Gallouet_2002}. 
Conversely, letting $v_K=1$ and $v_\sigma=0$ in \eqref{varfor-scheme} leads to  \eqref{strong-bal} and letting $v_K=0$ and $v_\sigma=1$ in \eqref{varfor-scheme} leads to \eqref{strong-cons}. 
The condition \eqref{strong-cl} is ensured by the fact that $u \in \Xmesh$.
\end{proof}

 Note that the inflated approximate gradient is expected to weakly converge in $L^2$ toward the gradient of the exact solution but can never converge in $L^2$ toward the gradient of the exact solution except if the exact solution is $0$ (see \cite[Lemma 2 and Remark 2]{ar4}). The convergence in $L^2$ of the consistent approximate gradient is proved in \cite{vf100}, in the case where $\bF = 0$, and using a modified finite volume scheme in order to handle anisotropic diffusion problems, under the condition that $\theta_{\mesh}$ is uniformly bounded by below.

\section{Convergence analysis}

\subsection{An optimal error estimate}

 A norm on $\Xmesh$ is defined by 
 \begin{equation}
  \label{eqdef:normHM}
   \Vert u\Vert_{\mesh}^2 = d \Vert G_{\mesh}u\Vert_{L^2(\Omega)}^2 = \frac 1 d \Vert \nabla_{\mesh}u\Vert_{L^2(\Omega)^d}^2 = \sum_{K\in \mesh} \sum_{\sigma \in \edges_K} \frac{|\sigma|}{d_{K,\sigma}} (u_\sigma - u_K)^2.
\end{equation}

The discrete Poincaré inequality \cite[Lemma 9.1]{book} states that for any piecewise function $u$ equal to $u_K$ on the cell $K$, 
 \begin{equation}
   \Vert u\Vert_{L^2(\Omega)} \le  {\rm diam}(\Omega) \Big(\sum_{\sigma = K|L \in \edgesint} \frac{|\sigma|}{d_{K,\sigma}+d_{L,\sigma}} (u_K - u_L)^2 +\sum_{\sigma  \in \edgesext\cap \edges_K} \frac{|\sigma|}{d_{K,\sigma}} u_K ^2\Big).
 \end{equation}
Now consider $u\in \Xmesh$ satisfying
 \begin{equation}\label{eq:usigma}
  \frac {u_\sigma - u_K} {d_{K,\sigma}} + \frac {u_\sigma - u_L} {d_{L,\sigma}} = 0,\hbox{ for all } \sigma  = K|L \in \edgesint;
  \end{equation}
this yields $\frac{(u_\sigma - u_K)^2}{d_{K,\sigma}}+\frac{(u_\sigma - u_L)^2}{d_{L,\sigma}}  = \frac{(u_L - u_K)^2}{d_{K,\sigma}+d_{L,\sigma}}$ and therefore, recalling that $u_\sigma =0$ for any $\sigma  \in \edgesext$,
\[
 \sum_{\sigma = K|L \in \edgesint} \frac{|\sigma|}{d_{K,\sigma}+d_{L,\sigma}} (u_K - u_L)^2 +\sum_{\sigma  \in \edgesext\cap \edges_K} \frac{|\sigma|}{d_{K,\sigma}} u_K ^2= \sum_{K\in \mesh} \sum_{\sigma \in \edges_K} \frac{|\sigma|}{d_{K,\sigma}} (u_\sigma - u_K)^2.
\]
Observe then that the minimal value of the function $s \mapsto \frac {(s - u_K)^2} {d_{K,\sigma}} + \frac {(s - u_L)^2} {d_{L,\sigma}}$ is obtained for $s = u_\sigma$ such that \eqref{eq:usigma} holds, so that for any $u\in X_\mesh$, whatever the values $u_\sigma$ for  any $\sigma  = K|L \in \edgesint$,  
 \[
    \Vert u\Vert_{\mesh}^2 \ge \sum_{\sigma = K|L \in \edgesint} \frac{|\sigma|}{d_{K,\sigma}+d_{L,\sigma}} (u_K - u_L)^2+\sum_{\sigma  \in \edgesext\cap \edges_K} \frac{|\sigma|}{d_{K,\sigma}} u_K ^2.
 \]
Hence
 \begin{equation}\label{eq:poindis}
 \Vert u\Vert_{L^2(\Omega)} \le {\rm diam}(\Omega) \Vert u\Vert_{\mesh} = {\rm diam}(\Omega) \sqrt{d}\Vert G_{\mesh}u\Vert_{L^2(\Omega)}, \ \forall u\in \Xmesh.
 \end{equation}
 Note that \eqref{eq:usigma} expresses the conservativity of the discrete fluxes \eqref{strong-cons} in the case $\bF=0$.

Next, we recall that the space $\hdiv(\Omega) =\{{\bm\varphi}\in L^2(\Omega)^d: \ \div {\bm\varphi} \in L^2(\Omega)\}$ is a Hilbert space when equipped with the norm
\[
 \Vert {\bm\varphi}\Vert_{\hdiv}^2 =  \Vert {\bm\varphi}\Vert_{L^2(\Omega)^d}^2 +  \Vert \div{\bm\varphi}\Vert_{L^2(\Omega)}^2.
\]

We introduce the conformity error: $\zeta_{\mesh}: \hdiv(\Omega)\to \mathbb{R}$, defined by
\begin{equation}\label{eq:defconfer}
 \zeta_{\mesh}({\bm\varphi}) = \sup\{ \int_\Omega ({\rm div}{\bm\varphi}(x)v(x) +{\bm\varphi}(x)\cdot \nabla_{\mesh}v(x)) {\rm d}x,  \; v\in \Xmesh\hbox{ with }\Vert v\Vert_{\mesh} = 1\}.
 \end{equation}

Let $\gnc_{\mesh}:H_0^1(\Omega)\to L^2(\Omega)$, called in the following the {\it mean normal gradient} of a function belonging to $H_0^1(\Omega)$, defined for $\varphi\in H^1_0(\Omega)$ by
\begin{equation}\label{eq:defogm}
 \gnc_{\mesh}\varphi(x) = \frac 1 {|D_{K,\sigma}|}\int_{D_{K,\sigma}}\nabla\varphi(x)\cdot {\bm n}_{K,\sigma}{\rm d}x \text{ for a.e. } x\in D_{K,\sigma},
 \end{equation}
 Observe that Definition \eqref{eq:defogm} implies the following equality, for $\overline{u} \in H^1_0(\Omega)$ and $v \in H_\mesh$,
 \begin{equation}
  \label{eq:gradgradM-GMGM}
  \int_\Omega \nabla \overline{u}\cdot \nabla_{\mesh}v(x){\rm d}x =  d \int_\Omega \gnc_{\mesh}\overline{u}(x)G_{\mesh}v(x)  {\rm d}x.
 \end{equation}
Next, we define a distance between any function $v\in \Xmesh$ and $\varphi\in H^1_0(\Omega)$ in the following way:
\begin{equation}\label{eq:definter}
 \delta_{\mesh}(\varphi,v) =\frac 1 {{\rm diam}(\Omega)} \Vert   \varphi- v\Vert_{L^2(\Omega)} + \sqrt{d}\Vert \gnc_{\mesh}\varphi - G_{\mesh}v\Vert_{L^2(\Omega)}. 
\end{equation}
We may then define a (generalized) interpolation error by:
\begin{equation}\label{eq:definterpol}
 \forall \varphi\in H^1_0(\Omega),\ \mathcal I_{\mesh}(\varphi) = \inf_{v\in \Xmesh}\delta_{\mesh}(\varphi,v). 
\end{equation}
The following theorem yields an optimal error bound, in the sense that the distance between the solution $\bar u$ to \eqref{eq:laplace} and the solution $u$ to \eqref{eq:fvscheme} is bounded by above and below by the sum of the conformity error and the interpolation error.

\begin{thm}[Optimal error bound for the approximation of the elliptic problem  \eqref{eq:pbellcont}] \label{thm:errest}
 Let $\overline{u}$ be the solution to the elliptic problem \eqref{eq:pbellcont}. 
 Then the solution $u$ to the numerical scheme \eqref{eq:fvscheme} satisfies
\begin{equation}\label{eq:errest}
 \frac 1 2 \Big(\zeta_{\mesh}(\nabla\overline{u}+\bF) + \mathcal I_{\mesh}(\overline{u})\Big)
 \le \delta_{\mesh}(\overline{u},u) 
 \le 3 \Big(\zeta_{\mesh}(\nabla\overline{u}+\bF) + \mathcal I_{\mesh}(\overline{u})\Big).
\end{equation}
\end{thm}
\begin{remark}[On the definition of $\delta_{\mesh}(\varphi,v)$]\label{rem:compdelta}
 If we replace the definition \eqref{eq:definter} of $\delta_{\mesh}$ by
\[
 \delta_{\mesh,w}(\varphi,v) =\frac 1 {{\rm diam}(\Omega)} \Vert  w- v\Vert_{L^2(\Omega)} + \sqrt{d}\Vert \gnc_{\mesh}\varphi - G_{\mesh}v\Vert_{L^2(\Omega)},
 \]
 with $w \in L^2(\Omega)$ possibly depending on $\mesh$, then Theorem \ref{thm:errest} still holds. 
 Nevertheless, in order for the bound \eqref{eq:errest} to yield an error estimate,   
 the interpolation error $\mathcal I_{\mesh}(\overline{u})$ must tend to 0 with the size of the mesh. 
 This implies that $w$ tends to $\overline{u}$ in $L^2(\Omega)$ with the size of the mesh; this is indeed the case if, for instance, $w$ is the piecewise constant function defined by the mean value over the cells, or, for regular enough functions, by the value of $\overline{u}$ at point $x_K$ (for the cell $K$).
 
 \end{remark}
 \begin{remark}[On the definition of  $\gnc_{\mesh}$]\label{rem:defgnc}
 The results of  Theorem \ref{thm:errest} remain true if we define  $\gnc_{\mesh}:H_0^1(\Omega)\to L^2(\Omega)$ by
\begin{equation}\label{eq:defogmbis}
 \gnc_{\mesh}\overline{u}(x) = \nabla\overline{u}(x)\cdot {\bm n}_{K,\sigma} \text{ for a.e. } x\in D_{K,\sigma},
 \end{equation}
 that is by the value of the normal gradient instead of its mean value in $D_{K,\sigma}$. 
However, this choice leads to larger values for $ \delta_{\mesh}(\overline{u},u)$.
Moreover, it leads to expressions which can be more difficult to evaluate  in the numerical implementation.
 \end{remark}

\begin{proof}[Proof of Theorem \ref{thm:errest}]
From the definition \eqref{eq:defconfer} of the conformity error $\zeta_\mesh$ and owing to \eqref{dans:hdiv}, we obtain that
 \[
 \int_\Omega \left((\nabla\overline{u}+\bF)(x)\cdot \nabla_{\mesh}v(x)-f(x)v(x)\right) {\rm d}x\le  \zeta_{\mesh}(\nabla\overline{u}+\bF) \Vert v\Vert_{\mesh}, \; \forall v\in  \Xmesh.
 \]
 Since $u$ is the solution to \eqref{eq:fvscheme}, we get
\[
 \int_\Omega (\nabla\overline{u}(x)\cdot \nabla_{\mesh}v(x)-d\, G_{\mesh}u(x)G_{\mesh}v(x)) {\rm d}x\le  \zeta_{\mesh}(\nabla\overline{u}+\bF) \Vert v\Vert_{\mesh},\; \forall v\in  \Xmesh.
 \]
 Owing to \eqref{eq:gradgradM-GMGM}, we have
 \[
 \int_\Omega d \,(\gnc_{\mesh}\overline{u}(x)- G_{\mesh}u(x))G_{\mesh}v(x) {\rm d}x=\int_\Omega d \,(\gnc_{\mesh}\overline{u}(x)- G_{\mesh}u(x))G_{\mesh}v(x) {\rm d}x\le  \zeta_{\mesh}(\nabla\overline{u}+\bF) \Vert v\Vert_{\mesh}.
 \]
 This inequality, together with the triangle and the Cauchy-Schwarz inequalities yield that for any $w\in  \Xmesh$, 
 \begin{multline*}
 \int_\Omega d (G_{\mesh}w(x) - G_{\mesh}u(x))G_{\mesh}v(x) {\rm d}x\\ 
 \le  \zeta_{\mesh}(\nabla\overline{u}+\bF) \Vert v\Vert_{\mesh} + d \Vert \gnc_{\mesh}\overline{u} - G_{\mesh}w\Vert_{L^2(\Omega)}\Vert  G_{\mesh}v\Vert_{L^2(\Omega)},\; \forall v\in  \Xmesh,
 \end{multline*}
Choosing $v = w - u$ and simplifying by $\sqrt{d}\Vert  G_{\mesh}v\Vert_{L^2(\Omega)}$, we get
\begin{equation}\label{inter}
   \sqrt{d} \Vert G_{\mesh}u - G_{\mesh}w\Vert_{L^2(\Omega)}
 \le  \zeta_{\mesh}(\nabla\overline{u}+\bF)   +  \sqrt{d} \Vert \gnc_{\mesh}\overline{u} - G_{\mesh}w\Vert_{L^2(\Omega)},\; \forall w\in  \Xmesh.
\end{equation}
 Owing to the Poincaré inequality \eqref{eq:poindis}, this latter inequality implies that
 \[
   \frac 1 {{\rm diam}(\Omega)}  \Vert u - w\Vert_{L^2(\Omega)}
 \le \zeta_{\mesh}(\nabla\overline{u}+\bF)   +  \sqrt{d}  \Vert \gnc_{\mesh}\overline{u} - G_{\mesh}w\Vert_{L^2(\Omega)},
 \]
 and thanks to the triangle inequality,
 \begin{multline*}
 \frac 1 {{\rm diam}(\Omega)} \Vert u - \overline{u}\Vert_{L^2(\Omega)}
 \le \zeta_{\mesh}(\nabla\overline{u}+\bF)   +  \sqrt{d}\Vert \gnc_{\mesh}\overline{u} - G_{\mesh}w\Vert_{L^2(\Omega)}+  \frac 1 {{\rm diam}(\Omega)} \Vert \overline{u}-w\Vert_{L^2(\Omega)}.
 \end{multline*}
Again by the triangle inequality, we get from \eqref{inter} that
\[
   \sqrt{d} \Vert G_{\mesh}u - \gnc_{\mesh}\overline{u}\Vert_{L^2(\Omega)}
 \le  \zeta_{\mesh}(\nabla\overline{u}+\bF)   +  2\sqrt{d} \Vert \gnc_{\mesh}\overline{u} - G_{\mesh}w\Vert_{L^2(\Omega)}.
 \]
Adding the two previous inequalities and taking the infimum on $w\in \Xmesh$ yields the inequality on the right of \eqref{eq:errest}.
 
 \medskip
 
 Let us now prove the left inequality of \eqref{eq:errest}. 
 Let $v\in  \Xmesh$; from \eqref{eq:gradgradM-GMGM} and the FV scheme \eqref{eq:fvscheme}, and owing to the definitions \eqref{eq:definter} of and \eqref{eqdef:normHM} of the norm, we get that,
 \begin{align*}
   \int_\Omega ((\nabla\overline{u}+\bF)(x)\cdot \nabla_{\mesh}v(x) -f(x)v(x)) {\rm d}x &= \int_\Omega d (\gnc_{\mesh}\overline{u}(x)- G_{\mesh}u(x))G_{\mesh}v(x) {\rm d}x\\
&\le  \delta_{\mesh}(\overline{u},u)  \Vert v\Vert_{\mesh}.
 \end{align*}
 Passing to the supremum on  the functions $v \in H_\mesh$ which are such that $\Vert v \Vert_\mesh = 1$ yields
 \[
   \zeta_{\mesh}(\nabla\overline{u}+\bF)  \le  \delta_{\mesh}(\overline{u},u),
 \]
 and since $
  \mathcal I_\mesh \overline{u} =  \inf_{v\in \Xmesh}\delta_{\mesh}(\overline{u},v) \le \delta_{\mesh}(\overline{u},u),
$
we get the  left part of \eqref{eq:errest}.

\end{proof}

 \begin{remark}[Existence and uniqueness]
Note that the existence and uniqueness result of $u$, solution to the numerical scheme \eqref{eq:fvscheme} which was proven in \cite[Theorem 2.1]{Droniou_Gallouet_2002}, may also be seen as a consequence of \eqref{eq:errest}. 
Indeed, the components of $u$ are solution to the square linear system given by the numerical scheme. 
A null right hand side to this linear system is obtained by setting $\bF = 0$ and $f =0$, which leads to $\overline{u} = 0$. 
The error estimate result yields a $0$ value in the right hand side of \eqref{eq:errest} and therefore $\delta_{\mesh}(0,u)=0$, implying $u = 0$. 
Hence the linear system is invertible so that there exists a unique solution to \eqref{eq:fvscheme}. 
 \end{remark}

\begin{lem}[Error bound for the consistent approximate gradient]\label{lem:cvgradcons}
 Let $\overline{u}\in H^1_0(\Omega)$ and let $u\in  \Xmesh$, then the following bound holds:
 \[
  \Vert \widehat\nabla_{\mesh} u - \nabla \overline{u}\Vert_{L^2(\Omega)^d} \le \frac {d}{\theta_{\mesh}} (\Vert G_{\mesh} u - \gnc_{\mesh} \overline{u}\Vert_{L^2(\Omega)} + \Theta_{\mesh}(\nabla\overline{u})),
 \]
 with, for any ${\bm\varphi}\in L^2(\Omega)^d$,
 \begin{equation}\label{eq:deftphi}
  \Theta_{\mesh}({\bm\varphi})^2 = \sum_{K\in\mesh}  \int_K  \int_K \frac 1 {|K|}|{\bm\varphi}(y) - {\bm\varphi}(x)|^2{\rm d}x {\rm d}y.
 \end{equation}
As a consequence, if $\overline{u}$ is the solution to \eqref{eq:pbellcont} and $u$ is the solution of \eqref{eq:fvscheme}, then 
 \begin{equation}\label{eq:errestgrad}
 \Vert \widehat\nabla_{\mesh} u - \nabla \overline{u}\Vert_{L^2(\Omega)^d} \le \frac {d}{\theta_{\mesh}}\Big( 2(1+ {\rm diam}(\Omega)) \big(\zeta_{\mesh}(\nabla\overline{u}+\bF) + \inf_{v\in \Xmesh}\delta_{\mesh}(\overline{u},v)\big) + \Theta_{\mesh}(\nabla\overline{u})\Big).
\end{equation}
\end{lem}
\begin{proof}
 We define $\overline{\nabla}_{K,\sigma}\overline{u} =\frac 1 {|D_{K,\sigma}|}\int_{D_{K,\sigma}} \nabla\overline{u}(x){\rm d}x $. 
 Note that $\gnc_{\mesh} \overline{u}$ is a.e. equal to ${\bm n}_{K,\sigma}\cdot\overline{\nabla}_{K,\sigma}\overline{u}$ in $D_{K,\sigma}$.
Owing to  \cite[Lemme 2.4]{vf100},  
 \[
  \frac 1 {|K|}\sum_{\sigma\in\mathcal{E}_K} |\sigma|(\overline{x}_\sigma - x_K) {\bm n}_{K,\sigma}^t = {\rm Id}_d,
 \]
 where ${\rm Id}_d$ is the $d\times d$ identity matrix.
 This is true for any choice of $x_K$, however in the sequel we use it for the choice of $x_K$ satisfying the definition \ref{def:meshdirichlet} of admissible meshes.
Let $\overline\nabla_{\mesh} \overline{u}$ be defined a.e. on $\Omega$ by:
 \[
 \forall x \in K, \; \overline\nabla_{\mesh} \overline{u}(x) \overline\nabla_K \overline{u} = \frac 1 {|K|} \sum_{\sigma\in\mathcal{E}_K} |\sigma|(\overline{x}_\sigma - x_K) {\bm n}_{K,\sigma}\cdot\overline{\nabla}_{K,\sigma}\overline{u},
 \]
By definition  of the mesh regularity parameter $\theta_mesh$, we have  $|\overline{x}_\sigma - x_K| \le \frac{d_{K,\sigma}}{\theta_{\mesh}}$,
 \begin{align*}
   \Vert\overline\nabla_{\mesh} \overline{u} - \nabla \overline{u}\Vert_{L^2(\Omega)^d}^2 &= \sum_{K\in\mesh} \int_K\Big(\frac 1 {|K|} \sum_{\sigma\in\mathcal{E}_K} |\sigma|(\overline{x}_\sigma - x_K) {\bm n}_{K,\sigma}\cdot(\overline{\nabla}_{K,\sigma}\overline{u}  -\nabla\overline{u}(x)\Big)^2 {\rm d}x \\
   &\le  \frac {d^2}{\theta_{\mesh}^2}\sum_{K\in\mesh} \int_K\Big(\frac 1 {|K|} \sum_{\sigma\in\mathcal{E}_K}  |D_{K,\sigma}| |\overline{\nabla}_{K,\sigma}\overline{u}  -\nabla\overline{u}(x)|\Big)^2 {\rm d}x\\
   &\le \frac {d^2}{\theta_{\mesh}^2}\sum_{K\in\mesh} \int_K \frac 1 {|K|} \sum_{\sigma\in\mathcal{E}_K} |D_{K,\sigma}| \Big|\overline{\nabla}_{K,\sigma}\overline{u}  -\nabla\overline{u}(x)\Big|^2{\rm d}x.
 \end{align*}
On the one hand, by the Cauchy-Schwarz inequality,
 \[
   \Big|\overline{\nabla}_{K,\sigma}\overline{u}  -\nabla\overline{u}(x)\Big|^2   =
   \Big| \frac 1 {|D_{K,\sigma}|}\int_{D_{K,\sigma}} (\nabla\overline{u}(y) - \nabla\overline{u}(x)){\rm d}y\Big|^2   \le
   \frac 1 {|D_{K,\sigma}|} \int_{D_{K,\sigma}} |\nabla\overline{u}(y) - \nabla\overline{u}(x)|^2{\rm d}y,
\]
so that
\[
   \Vert\overline\nabla_{\mesh} \overline{u} - \nabla \overline{u}\Vert_{L^2(\Omega)^d}^2  
   \le \frac {d^2}{\theta_{\mesh}^2}\sum_{K\in\mesh} \sum_{\sigma\in\mathcal{E}_K} \int_{D_{K,\sigma}}  \int_K \frac 1 {|K|}|\nabla\overline{u}(y) - \nabla\overline{u}(x)|^2{\rm d}x {\rm d}y    
   =\frac {d^2}{\theta_{\mesh}^2}\Theta_{\mesh}(\nabla\overline{u})^2.
\]
On the other hand,  
 \begin{align*}
  \Vert \widehat\nabla_{\mesh} u - \overline\nabla_{\mesh} \overline{u}\Vert_{L^2(\Omega)^d}^2 & = \sum_{K\in\mesh}  \frac 1 {|K|} \Big|\sum_{\sigma\in\mathcal{E}_K} |\sigma|(\overline{x}_\sigma - x_K) (\frac {u_\sigma - u_K} {d_{K,\sigma}} - {\bm n}_{K,\sigma}\cdot\overline{\nabla}_{K,\sigma}\overline{u})\Big|^2\\ 
  & \le \sum_{K\in\mesh}  \frac 1 {|K|} \Big(\sum_{\sigma\in\mathcal{E}_K} |\sigma| |\overline{x}_\sigma - x_K| \big|\frac {u_\sigma - u_K} {d_{K,\sigma}} - {\bm n}_{K,\sigma}\cdot\overline{\nabla}_{K,\sigma}\overline{u})\big|\Big)^2
  \\ 
  &\le  \frac {d^2}{\theta_{\mesh}^2}
  \sum_{K\in\mesh}  \frac 1 {|K|} \Big(\sum_{\sigma\in\mathcal{E}_K}  |D_{K,\sigma}|  \big|\frac {u_\sigma - u_K} {d_{K,\sigma}} - {\bm n}_{K,\sigma}\cdot\overline{\nabla}_{K,\sigma}\overline{u})\big|\Big)^2
 \end{align*}
 Applying the Cauchy-Schwarz inequality, we obtain
 \begin{align*}
  \Vert \widehat\nabla_{\mesh} u - \overline\nabla_{\mesh} \overline{u}\Vert_{L^2(\Omega)^d}^2 &\le  \frac {d^2}{\theta_{\mesh}^2}\sum_{K\in\mesh}  \sum_{\sigma\in\mathcal{E}_K} |D_{K,\sigma}| (\frac {u_\sigma - u_K} {d_{K,\sigma}} - {\bm n}_{K,\sigma}\cdot\overline{\nabla}_{K,\sigma}\overline{u})^2\\ &\le  \frac {d^2}{\theta_{\mesh}^2}\sum_{K\in\mesh}  \sum_{\sigma\in\mathcal{E}_K} |D_{K,\sigma}| (\frac {u_\sigma - u_K} {d_{K,\sigma}} - {\bm n}_{K,\sigma}\cdot\overline{\nabla}_{K,\sigma}\overline{u})^2\\
  &=\Big(\frac {d}{\theta_{\mesh}} \Vert G_{\mesh} u - \gnc_{\mesh} \overline{u}\Vert_{L^2(\Omega)}\Big)^2.
 \end{align*}
 \end{proof}

 \subsection{Convergence of the interpolation and conformity errors}
Owing to Theorem \ref{thm:errest}, the convergence of the scheme \eqref{eq:cvfvscheme} relies on the convergence of the conformity error  and the interpolation error. 

\begin{lem}[Convergence of the conformity error]\label{lem:cvzeta}
 Let $(\mesh_n)_{n\in\mathbb{N}}$ be a sequence of admissible meshes such that $h_{\mesh_n} = \max_{K\in \mesh_n} {\rm diam}(K)$ tends to 0 as $n\to\infty$.
 For any ${\bm\varphi}\in \hdiv(\Omega)$,
 \[
  \lim_{n\to\infty}\zeta_{\mesh_n}({\bm\varphi}) = 0.
 \]
\end{lem}
\begin{proof}
We first consider the case where ${\bm\varphi}\in C^1(\overline{\Omega})^d$. 
Let us momentarily drop the index $n$ for simplicity, and let $v\in \Xmesh$ with $\Vert v\Vert_{\mesh} = 1$.
Owing to the definition \eqref{eqdef:conformity} of $\zeta_{\mesh}$, 
\[
\zeta_{\mesh}({\bm\varphi})\!\! = \!\!\!\!\sup_{\substack{v\in \Xmesh\\\Vert v\Vert_{\mesh} = 1}}  \!\!\!\!z_{\mesh}({\bm\varphi},v) 
\mbox{ where  } \, z_{\mesh}({\bm\varphi},v)  =\!\!\int_\Omega\!\! ({\rm div}{\bm\varphi}(x)v(x)  +{\bm\varphi}(x)\cdot \nabla_{\mesh}v(x)) {\rm d}x.  
\]
 Let
$
 {\bm\varphi}_{K,\sigma} =  \frac 1 {|D_{K,\sigma}|}\int_{D_{K,\sigma}}{\bm\varphi}(x) {\rm d}x
\mbox{ and }   {\bm\varphi}_{\sigma} =  \frac 1 {|\sigma|}\int_{\sigma}{\bm\varphi}(x) \dfrontiere ;
$
by the definition \ref{eq:defnablam} of $\nabla_{\mesh}$, we have 
\[
 z_{\mesh}({\bm\varphi},v)   = \sum_{K\in\mesh}v_K \sum_{\sigma\in\mathcal{E}_K}|\sigma|\varphi_\sigma\cdot {\bm n}_{K,\sigma} + \sum_{K\in\mesh} \sum_{\sigma\in\mathcal{E}_K} |D_{K,\sigma}|   (v_\sigma-v_K)  {\bm\varphi}_{K,\sigma}\cdot {\bm n}_{K,\sigma}.
\]
 Observing that $|D_{K,\sigma}|  = \dfrac {|\sigma| d_{K,\sigma}} d $ and that $\displaystyle\sum_{K\in\mesh} \sum_{\sigma\in\mathcal{E}_K}|\sigma| v_\sigma  {\bm\varphi}_{K,\sigma}\cdot {\bm n}_{K,\sigma} =0$, we get that 
 \begin{align*}
  z_{\mesh}({\bm\varphi},v) & =     \sum_{K\in\mesh}\sum_{\sigma\in\mathcal{E}_K}|\sigma|(v_\sigma-v_K)\big( {\bm\varphi}_{K,\sigma}\cdot {\bm n}_{K,\sigma} -  {\bm\varphi}_{\sigma}\cdot {\bm n}_{K,\sigma} \big)\\
  & \le  \Vert v\Vert_{\mesh}\big( \sum_{K\in\mesh}\sum_{\sigma\in\mathcal{E}_K}|\sigma| d_{K,\sigma}|{\bm\varphi}_{K,\sigma} -  {\bm\varphi}_{\sigma}|^2\big)^{1/2},
 \end{align*}
owing to the Cauchy-Schwarz inequality.
Hence
\[
 \zeta_{\mesh}({\bm\varphi}) \le C_{\bm\varphi} h_{\mesh} (d\ |\Omega|)^{1/2}.
\]
We now consider the general case   ${\bm\varphi}\in \hdiv(\Omega)$. Let $\bm\psi\in C^1(\overline{\Omega})^d$ be given.
Again dropping the subscript $n$, we have for $v\in \Xmesh$ with $\Vert v\Vert_{\mesh} = 1$
 \begin{align*}
  \int_\Omega ({\rm div}{\bm\varphi}(x)v(x) +{\bm\varphi}(x)\cdot \nabla_{\mesh}v(x)) {\rm d}x &= \int_\Omega ({\rm div}\bm\psi(x)v(x) +\bm\psi(x)\cdot \nabla_{\mesh}v(x)) {\rm d}x + \delta \\
  & \le  C_{\bm\psi} h_{\mesh} (d\ |\Omega|)^{1/2} + \delta,
\end{align*}
where thanks to  \eqref{eq:poindis} and to the Cauchy-Schwarz inequality,
 \begin{multline*}
  |\delta| \le \Vert {\bm\varphi} - \bm\psi\Vert_{L^2(\Omega)^d}\Vert  \nabla_{\mesh}v\Vert_{L^2(\Omega)^d} + \Vert {\rm div}({\bm\varphi} - \bm\psi)\Vert_{L^2(\Omega)}\Vert  v\Vert_{L^2(\Omega)} \\
  \le (d+{\rm diam}(\Omega)^2)^{1/2}\Vert {\bm\varphi} - \bm\psi\Vert_{\hdiv(\Omega)}.
\end{multline*}
Hence
\[
 \zeta_{\mesh}({\bm\varphi}) \le C_{\bm\psi} h_{\mesh} (d\ |\Omega|)^{1/2} +  (d+{\rm diam}(\Omega)^2)^{1/2}\Vert {\bm\varphi} - \bm\psi\Vert_{\hdiv(\Omega)}.
\]
We conclude thanks to \cite[Theorem 1.1]{temam1984theory} which states the density of $C^1(\overline{\Omega})^d$ in $\hdiv(\Omega)$ for any Lipschitz open set $\Omega$ of $\mathbb{R}^d$.
\end{proof}

\begin{lem}[Convergence of the interpolation error]\label{lem:cvdelta}
Let $(\mesh_n)_{n\in\mathbb{N}}$ be a sequence of admissible meshes such that $h_{\mesh_n} = \max_{K\in \mesh_n} {\rm diam}(K)$ tends to 0 as $n\to\infty$. 
Then, for any $\varphi\in H_0^1(\Omega)$,
 \[
  \lim_{n\to\infty} \mathcal{I}_{\mesh_n}(\varphi).
 \]

\end{lem}
\begin{proof}
We first consider $\varphi\in C^\infty_c(\Omega)$. 
Let  $v^{(n)}\in H_{\mesh_n}$ be defined by $v^{(n)}_K = \varphi(x_K)$ and $v^{(n)}_\sigma$ be such that 
\[
  |\sigma|\frac {v^{(n)}_\sigma - v^{(n)}_K} {d_{K,\sigma}} + |\sigma|\frac {v^{(n)}_\sigma - v^{(n)}_L} {d_{L,\sigma}} = 0,\hbox{ for all } \sigma\in \mathcal{E}_K\cap \mathcal{E}_L,
\] 
By the definition \eqref{eq:definter},
\[
 \delta_{\mesh_n}(\varphi,v^{(n)}) = \frac {1} {{\rm diam}(\Omega)} \Vert  \varphi- v^{(n)}\Vert_{L^2(\Omega)} + \sqrt{d}\Vert \gnc_{\mesh_n}\varphi - G_{\mesh_n}v^{(n)}\Vert_{L^2(\Omega)}.
 \]
 It is clear that $\Vert  \varphi- v^{(n)}\Vert_{L^2(\Omega)}\to 0 $ as $\nti$.
 Now observe that
 \[
\frac {v^{(n)}_\sigma - v^{(n)}_K} {d_{K,\sigma}} = \frac{v^{(n)}_L-v^{(n)}_K} {d_{K,\sigma}+d_{L,\sigma}};
 \]
therefore, since $\varphi\in C^\infty_c(\Omega)$, we have for any $x\in D_{K,\sigma}$,
\[
 |\gnc_{\mesh_n}\varphi(x) - G_{\mesh_n}v^{(n)}(x)|\le  C_{\varphi} h_{\mesh_n},
 \]
and the result follows. 

\smallskip

We now consider the case $\varphi\in H_0^1(\Omega)$; let  $\psi\in C^\infty_c(\Omega)$ be given. 
By the triangle inequality, 
\[
 \delta_{\mesh_n}(\varphi,v)\le  \delta_{\mesh_n}(\psi,v) + \frac{1}{\mathrm{diam}(\Omega)}\Vert\varphi -\psi\Vert_{L^2} +\sqrt{d}\Vert\gnc_{\mesh_n}\varphi -\gnc_{\mesh_n}\psi\Vert_{L^2}.
\]
Observing that for a.e. $x\in \Omega$,
\[
 |\gnc_{\mesh_n}(\varphi- \psi)(x)|\le | \nabla(\varphi- \psi)(x)|,
\]
we get 
\[
 \delta_{\mesh_n}(\varphi,v)\le  \delta_{\mesh_n}(\psi,v) + \max(\frac{1}{\mathrm{diam}(\Omega)},\sqrt{d})\Vert\varphi -\psi\Vert_{H^1},
\]
and the result follows by density. 
\end{proof}

\begin{lem}\label{lem:cvtphi}
Let ${\bm\varphi}\in L^2(\Omega)^d$, and let  $\Theta_{\mesh_n}({\bm\varphi})$ be defined by \eqref{eq:deftphi} for any $n\in\mathbb{N}$.
Then
 \begin{equation}\label{eq:limtphi}
  \lim_{n\to\infty} \Theta_{\mesh_n}({\bm\varphi}) = 0.
 \end{equation}

\end{lem}
\begin{proof}
We observe that, for a given admissible mesh $\mesh$ and for $\bm\psi\in C^1(\overline{\Omega})^d$, using
\[
 |\bm\psi(y) - \bm\psi(x)|\le h_{\mesh} C_{\bm\psi},
\]
we get
\[
 |\Theta_{\mesh}(\bm\psi)|\le h_{\mesh} C_{\bm\psi}\sqrt{|\Omega|}.
\]
By the triangle inequality
\[
 |{\bm\varphi}(y) - {\bm\varphi}(x)|\le |{\bm\varphi}(y) - \bm\psi(y)|+|\bm\psi(y) - \bm\psi(x)|+|\bm\psi(x) - {\bm\varphi}(x)|,
\]
we obtain
\[
 \Theta_{\mesh}({\bm\varphi})^2 \le 3 (2\Vert {\bm\varphi} - \bm\psi\Vert_{L^2(\Omega)^d}^2+ \Theta_{\mesh}(\bm\psi)^2)
\]
which concludes the proof by density of $C^1(\overline{\Omega})^d$ in $L^2(\Omega)^d$.
\end{proof}
We can then conclude the convergence of the finite volume scheme.
\begin{thm}[Convergence of the approximate solution and gradient to the solution of \eqref{eq:pbellcont}]\label{thm:cv}
 Let $\overline{u}\in H^1_0(\Omega)$ be the solution of \eqref{eq:pbellcont} and let $u_n\in H_{\mesh_n}$ be the solution of \eqref{eq:fvscheme} for $\mesh = \mesh_n$. 
 Then the following holds:
 \begin{equation}\label{eq:cvfvscheme}
  \lim_{n\to\infty} ( \Vert \overline{u}- u_n\Vert_{L^2(\Omega)}^2 + d\Vert \gnc_{\mesh_n}\overline{u} - G_{\mesh_n}u_n\Vert_{L^2(\Omega)}^2) = 0.
\end{equation}
Moreover, if the sequence $(\theta_{\mesh_n})_{n\in\mathbb{N}}$ is bounded by below by $\theta_0>0$, then
 \begin{equation}\label{eq:cvgradfvscheme}
  \lim_{n\to\infty}  \Vert \widehat\nabla_{\mesh_n} u_n - \nabla \overline{u}\Vert_{L^2(\Omega)} = 0.
\end{equation}
\end{thm}
\begin{proof}
Applying Lemmas \ref{lem:cvzeta} and \ref{lem:cvdelta}, \eqref{eq:cvfvscheme} is a consequence of Theorem \ref{thm:errest}. We conclude \eqref{eq:cvgradfvscheme} using Lemmas \ref{lem:cvtphi} and \ref{lem:cvgradcons}.
\end{proof}

 \subsection{Error estimate in the $H^2(\Omega)$ case}
 
We now suppose that the exact solution belongs to $H^2(\Omega)$.

\begin{lem}[Estimate of difference with average value in the $H^1$ case]\label{lem:diffhdeux}

 Let $V\subset \mathbb{R}^d$ be a bounded convex open set and let $h_V$  be the diameter of $V$. Let  $\varphi \in H^1(V)$. Then 
 \begin{equation}
\int_V \Big(\frac 1 {|V|} \int_V \varphi(y){\rm d}y - \varphi (x)\Big)^2 {\rm d}x \le  h_V^2  \frac {C_d  h_V^d} {|V|} \int_V |\grad \varphi (z)|^2 {\rm d}z,
\label{maj-varphi}
\end{equation}
where $C_d>0$ is the $d-$dimensional measure of the unit ball of $\R^d$.
\end{lem}
\begin{proof}
The proof of \eqref{maj-varphi} is inspired by that of \cite[Lemma 10.2]{book}, which is expressed in a discrete setting and is in fact trickier than the present proof. By density of $C^1_c(\R^d)$ in $H^1(V)$, we only have to prove \eqref{maj-varphi} for $\varphi \in C^1_c(\R^d)$. For $x$, $y \in V$,
\[
[\varphi(y)-\varphi(x)]^2=\big[\!\!\int_0^1 \!\!\grad \varphi(tx+(1-t)y)\cdot (x-y) {\rm d}t\big]^2 \!\!\le h_V^2\!\! \int_0^1 \!\!|\grad \varphi(tx+(1-t)y)|^2 {\rm d}t,
\]
so that owing to the Fubini-Tonelli theorem,
\begin{align*}
 \int_V \int_V [\varphi(y)-\varphi(x)]^2 {\rm d}x {\rm d}y &\le h_V^2\int_V  \int_V \int_0^1 |\grad \varphi(tx+(1-t)y)|^2 {\rm d}t {\rm d}x {\rm d}y
\\ 
& \le h_V^2 \int_V  \int_0^1 \int_V |\grad \varphi(tx+(1-t)y)|^2 {\rm d}x {\rm d}t {\rm d}y.
\end{align*}
Applying the change of variable (with $t$ and $y$ fixed) $z=tx+(1-t)y$,  noticing that $z \in V \cap B(y,t h_V)$,  and again applying the Fubini-Tonelli theorem, we thus get that 
\begin{align*}
\int_V \int_V (\varphi(y)-\varphi(x))^2 {\rm d}x {\rm d}y  
&\le  h_V^2 \!\!\int_V \!\bigl[ \int_0^1 \!\!\bigl(\!  \int_V\!\!  1_{B(z, th_V)}(y)  {\rm d}y \bigr)     t^{-d} {\rm d}t \bigr]   |\grad \varphi(z)|^2 {\rm d}z 
\\
&\le h_V^2 C_d h_V^d   \int_V  |\grad \varphi(z)|^2 {\rm d}z.
\end{align*}
The proof of  \eqref{maj-varphi} is thus complete.
\end{proof}

\begin{lem}[Estimate of $\Theta_{\mesh}(\nabla\overline{u})$ in the $H^2$ case]\label{lem:tmnablauhdeux}

 Let  $\overline{u}\in H^2(\Omega)$. 
 Then 
 \begin{equation}
\Theta_{\mesh}(\nabla\overline{u})\le  \frac {h_{\mesh}} {\theta_{\mesh}^{d/2}}  \Vert \overline{u}\Vert_{H^2}.
\label{maj-tmphi}
\end{equation}
\end{lem}
\begin{proof}
Let $K\in \mesh$ be a given control volume, and let $h_K$  be the diameter of $K$. 
By definition of $\theta_{\mesh}$, for any edge or face $\sigma$ of $K$,  $d_{K,\sigma}\ge \theta_{\mesh} h_K$; thus, the ball with center $x_K$ and radius $ \theta_{\mesh} h_K$ has $d-$dimensional measure $C_d  \times (\theta_{\mesh} h_K)^d$ and is included in $K$.
Now  $|K| \ge C_d \inf_{\sigma \in \mathcal{E}_K} d_{K, \sigma}^d$, and by definition of $\theta_{\mesh}$, $\inf_{\sigma \in  \mathcal{E}_K} d_{K, \sigma} \ge \theta_{\mesh} h_K$; therefore, $h_K^d \le \dfrac { |K|}{C_d \theta_{\mesh}^d}$; applying  \eqref{maj-varphi} to each of the components of $\nabla\overline{u}$ provides \eqref{maj-tmphi}.
\end{proof}

\begin{lem}[Error estimate in the $H^2$ case]\label{lem:errorhdeux}

 Let   $d\le 3$, $\overline{u}\in H^2(\Omega)$ and $\bF = 0$. Letting $u\in \Xmesh$ be defined by $u_K = \overline{u}(x_K)$,
 Then there exists $C$,  only depending on $d$, $\Omega$ and continuously depending on $\theta_{\mesh}$,  such that
 \[
  \Vert u - \overline{u}\Vert_{L^2(\Omega)} + \Vert \widehat\nabla_{\mesh} u - \nabla \overline{u}\Vert_{L^2(\Omega)^d} \le C h_{\mesh} \Vert \overline{u}\Vert_{H^2}.
 \]
\end{lem}
\begin{proof}
In this proof, we denote by $C_i$, for $i\in\mathbb{N}$, various real functions only depending on $d$, $\Omega$ and continuously depending on $\theta_{\mesh}$.
Let us first observe that
\begin{equation}\label{eq:erreur_interp}
 \inf_{v\in \Xmesh} \delta_{\mesh}(\overline{u},v)\le \cter{cte:c1}  h_{\mesh} \Vert \overline{u}\Vert_{H^2}.
\end{equation}
Indeed, following the computations of \cite[Theorem 9.4]{book} (which involves the regularity factor $\theta_{\mesh}$), and denoting $\widetilde{u}\in \Xmesh$ defined by $\widetilde{u}_K = \overline{u}(x_K)$ and $\widetilde{u}_\sigma$ such that 
\[
  |\sigma|\frac {\widetilde u_\sigma - \widetilde u_K} {d_{K,\sigma}} + |\sigma|\frac {\widetilde  u_\sigma - \widetilde u_L} {d_{L,\sigma}} = 0,\hbox{ for all } \sigma\in \mathcal{E}_K\cap \mathcal{E}_L,
\]
we get the existence of $\ctel{cte:c1}$ such that
\[
 \delta_{\mesh}(\overline{u},\widetilde{u})\le \cter{cte:c1}  h_{\mesh} \Vert \overline{u}\Vert_{H^2}.
\]
Since $\nabla\overline{u}\in H^1(\Omega)^d \subset \hdiv(\Omega)$, we have, for some $v\in \Xmesh$ with $\Vert v\Vert_{\mesh} = 1$, using the Cauchy-Schwarz inequality,
 \begin{align*}
  \int_\Omega ({\rm div}{\nabla\overline{u}}(x)v(x) +{\nabla\overline{u}}(x)\cdot \nabla_{\mesh}v(x)) {\rm d}x &= \sum_{K\in\mesh}\sum_{\sigma\in\mathcal{E}_K}|\sigma|(v_\sigma-v_K)\big( {\overline\nabla\overline{u}}_{K,\sigma}\cdot {\bm n}_{K,\sigma} -  {\overline\nabla\overline{u}}_{\sigma}\cdot {\bm n}_{K,\sigma} \big)\\
  &\le \Vert v\Vert_{\mesh}\big( \sum_{K\in\mesh}\sum_{\sigma\in\mathcal{E}_K}|\sigma| d_{K,\sigma}|{\overline\nabla\overline{u}}_{K,\sigma} -  {\overline\nabla\overline{u}}_{\sigma}|^2\big)^{1/2},
\end{align*}
where
\[
 {\overline\nabla\overline{u}}_{K,\sigma} =  \frac 1 {|D_{K,\sigma}|}\int_{D_{K,\sigma}}{\nabla\overline{u}}(x) {\rm d}x
\mbox{ and }{\overline\nabla\overline{u}}_{\sigma} =  \frac 1 {|\sigma|}\int_{\sigma}{\nabla\overline{u}}(x) \dfrontiere .
\]
Hence
\[
 \zeta_{\mesh}({\nabla\overline{u}}) \le \ctel{cte:c2} h_{\mesh}  \Vert \overline{u}\Vert_{H^2}.
\]
Using \eqref{maj-tmphi}, \eqref{eq:errest}, \eqref{eq:erreur_interp} and \eqref{eq:errestgrad}, the conclusion follows.
\end{proof}

\begin{remark}[Superconvergence]\label{rem:superconvergence}
 The superconvergence of the TPFA scheme on  a class of acute triangles in 2D was observed numerically long ago \cite{herbin-labergerie,BOIVIN2000806,2Dbench}, and was proved theoretically in \cite{droniou-nataraj} more recently; the proof of \cite{droniou-nataraj} relies on the fact that on a large class of 2D triangular grids, the TPFA scheme is a hybrid mixed method as defined in \cite{dro-10-uni}.
\end{remark}

\section{Numerical example of a minimal regularity problem}\label{sec:num}

We illustrate here the error estimate of Theorem \ref{thm:errest} in the case where the regularity of the solution of Problem \eqref{eq:pbellcont} is minimal.
Let $\Omega = (0,1)^2$, and let $\overline{u}\in H^1_0(\Omega)$ be defined by 
\begin{equation}\label{eq:solirr}
  \overline{u}(x) = \Big(-\log\big( \max(|x_1 - \frac 1 2|,|x_2 - \frac 1 2|)\big)\Big)^\gamma -  \Big(-\log( r_0 )\Big)^\gamma,\ x=(x_1,x_2)\in\Omega,
\end{equation}
with $\gamma = \frac 1 4$ and $r_0 = 1/2$. 
Then $\overline{u}$ is solution to \eqref{eq:pbellcont}, letting $f = 0$ and $\bF(x) = - \nabla\overline{u}(x)$. We have, for any $p>1$,
 \[
\int_\Omega|\nabla\overline{u}(x)|^p{\rm d}x = 4\int_0^{r_0} \int_{-x}^{x} \frac {\gamma^p} {x^p}\Big(-\log( x)\Big)^{p(\gamma-1)} {\rm d}y{\rm d}x = 8\int_0^{r_0} \frac {\gamma^p} {x^{p-1}}\Big(-\log( x)\Big)^{p(\gamma-1)} {\rm d}y{\rm d}x.
 \]
This integral is infinite for any $p>2$, which shows that $\overline{u} \not \in  W^{1,p}_0(\Omega)$ if $p>2$; this was expected since $\overline{u}$ is discontinuous and $W^{1,p}_0(\Omega)\subset C^0(\overline{\Omega})$ for any $p>2$ in 2D.
Now if $p=2$, we get, since $2\gamma-1<0$,
 \[
\int_\Omega|\nabla\overline{u}(x)|^2{\rm d}x = 8\int_0^{r_0}  \frac {\gamma^2} {x}\Big(-\log( x)\Big)^{2\gamma-2} {\rm d}x = \frac {8\gamma^2} {1 - 2\gamma}\Big(-\log(r_0)\Big)^{2\gamma-1},
 \]
which shows that $\nabla\overline{u}(x)\in L^2(\Omega)^2$ with $\Vert \nabla\overline{u}\Vert_{L^2} \simeq 1.0959573$. 
Let us now compute $\Vert \overline{u}\Vert_{L^2}$. We have
\[
 \int_\Omega\overline{u}(x)^2{\rm d}x = 4\int_0^{r_0} \int_{-x}^{x} \Big(-\log( x)\Big)^{2\gamma} {\rm d}y{\rm d}x - 2\Big(-\log(r_0)\Big)^\gamma 4\int_0^{r_0} \int_{-x}^{x} \Big(-\log( x)\Big)^{\gamma} {\rm d}y{\rm d}x+\Big(-\log(r_0)\Big)^{2\gamma},
\]
which gives
\[
 \int_\Omega\overline{u}(x)^2{\rm d}x = 8\int_0^{r_0} x \Big(-\log( x)\Big)^{2\gamma} {\rm d}x - 2\Big(-\log(r_0)\Big)^\gamma 8\int_0^{r_0} x \Big(-\log( x)\Big)^{\gamma} {\rm d}x+(2r_0)^2\Big(-\log(r_0)\Big)^{2\gamma}.
\]
Owing to the change of variable $-\log(x) = y/2$, we get
 \[
 \int_0^{r_0}  x\Big(-\log( x)\Big)^{2\gamma} {\rm d}x = \frac {1} {2^{2\gamma+1}}\int_{-2\log(r_0)}^\infty y^{2\gamma} e^{-y}{\rm d}y.
 \]
Introducing the upper incomplete gamma function $\Gamma$ defined by
 \[
  \Gamma(a,x) =  \int_x^{+\infty} t^{a - 1} e^{-t}{\rm d}t,
 \]
 we get
 \[
 \int_0^{r_0}  x\Big(-\log( x)\Big)^{2\gamma} {\rm d}x = \frac {1} {2^{2\gamma+1}}\Gamma(2\gamma+1,-2\log(r_0)). 
 \]
Again using the change of variable $-\log(x) = y/2$, we have
 \[
  \int_0^{r_0} x \Big(-\log( x)\Big)^{\gamma} {\rm d}x = \frac 1 {2^{\gamma+1}}\int_{-2\log(r_0)}^\infty y^{\gamma} e^{-y}{\rm d}y =  \frac 1 {2^{\gamma+1}}\Gamma(\gamma+1,-2\log(r_0))
 \]
Finally we obtain
 \begin{multline*}
 \int_\Omega\overline{u}(x)^2{\rm d}x = 2^{2 -2\gamma}\Gamma(2\gamma+1,-2\log(r_0))\\ - 2^{3-\gamma}\Big(-\log(r_0)\Big)^\gamma \Gamma(\gamma+1,-2\log(r_0))+(2r_0)^2\Big(-\log(r_0)\Big)^{2\gamma},
\end{multline*}
leading to $\Vert \overline{u}\Vert_{L^2} \simeq 0.1519926$.

\medskip

Let us then set $f = 0$ and $\bF(x) = - \nabla\overline{u}(x)$ in the finite volume scheme \eqref{eq:fvscheme}. 
Since $\nabla\overline{u}+\bF = 0$, we get that the conformity error $\zeta_{\mesh}(\nabla\overline{u}+\bF)$ vanishes. 
Hence the numerical validation of the error estimate of Theorem \ref{thm:errest} follows by a comparison of the approximation error $\delta_{\mesh}(\overline{u},u)$ and the upper bound of the interpolation error $\delta_{\mesh}(\overline{u},\overline{u}_{\mesh})$ for the interpolation $\overline u_\mesh\in X_{\mesh}$ of the exact solution $\overline{u}$, defined by the set of values $((\overline{u}(x_K))_{K\in\mesh},(\overline{u}(x_\sigma))_{\sigma\in\edges})$, where $x_\sigma$ is the middle of the edge $\sigma$ (it is then the intersection of the segment $[x_K,x_L]$ and of $\sigma$ when $\sigma$ is the common edge between $K$ and $L$).

\medskip

The right-hand side $\int_\Omega \nabla\overline{u}(x)\cdot \nabla_{\mesh}v(x) {\rm d}x$ can be computed by using the relation
 \[
  \int_{D_{K,\sigma}}\nabla\overline{u}(x)\cdot {\bm n}_{K,\sigma} {\rm d}x = \int_{\partial D_{K,\sigma}}\overline{u}(x) {\bm n}(y)\cdot {\bm n}_{K,\sigma} {\rm ds}(y),
 \]
 where ${\bm n}(y)$ is the outward unit vector normal to $\partial D_{K,\sigma}$ at the point $y$. 
 This latter integral is then a linear combination of terms under the form
 \begin{multline*}
\int_{0}^{1} \Big(-\log(\alpha + (\beta-\alpha) s)\Big)^\gamma {\rm d}s =  -\frac 1 {\beta -\alpha} \int_{-\log(\alpha)}^{-\log(\beta)} t^\gamma e^{-t} {\rm d}t \\ =\frac 1 {\beta -\alpha} \Big(\Gamma(\gamma+1,-\log(\beta))-\Gamma(\gamma+1,-\log(\alpha))\Big),
 \end{multline*}
Hence we can accurately compute the right-hand side for the computation of the approximate solution, even in the case where the elements have the point $(1/2,1/2)$ as vertex since we can extend by continuity the value of the function $g(\alpha) = \Gamma(a,-\log(\alpha))$  to 0 by setting $g(0) = 0$.
 
Note that the main part of $\Vert \nabla\overline{u}\Vert_{L^2(\Omega)^2}$ is concentrated inside a small ball. Indeed, denoting by $B_\infty(r)= \{ x\in\Omega, \max(|x_1 - \frac 1 2|,|x_2 - \frac 1 2|)<r\}$ (hence $\Omega = B_\infty(r_0)$), we get the following values for $\Vert \nabla\overline{u}\Vert_{L^2(B_\infty(r))^2}$ and $\Vert \overline{u}\Vert_{L^2(B_\infty(r))}$:

\medskip

\begin{tabular}{|c|c|c|c|c|c|}
  \hline
  $r$ & 0.5 & $\exp(-10)$ &  $\exp(-1000)$ &  $\exp(-1000000)$ &  $\exp(-1000000000)$   \\
  \hline
  $\Vert \nabla\overline{u}\Vert_{L^2(B_\infty(r))^2}$ & 1.096 & 0.562 & 0.178 & 0.032 & 0.006\\
  \hline
  $\Vert \overline{u}\Vert_{L^2(B_\infty(r))}$ & 0.152 & 2.71e-6 & < 1.e-200 & < 1.e-200000 & <1.e-200000000\\
  \hline
\end{tabular}

\medskip

It is clear from this table that $\lim_{r\to 0}\Vert \nabla\overline{u}\Vert_{L^2(B_\infty(r))^2} = 0$; however, this convergence is so slow that a numerical scheme based on a standard mesh cannot be expected to show a numerical convergence of the approximation of $\nabla\overline{u}$ for the $L^2$ norm, even with a refined mesh at the point $(0.5,0.5)$ (see Remark \ref{rem:compdelta}).  
Nevertheless, we can numerically assess the result proven in Theorem \ref{thm:errest} by considering the family of benchmark meshes \cite[Mesh1\_k]{2Dbench} with $k= 1,2,3,4,5,6,7$, which can be respectively characterized by the values $h_{\mesh} = \frac 1 4,\frac 1 {8},\frac 1 {16},\frac 1 {32},\frac 1 {64},\frac 1 {128},\frac 1 {256}$ for the size of the mesh. 
The following table gives the number of control volumes (cv), vertices and edges for each mesh.

\begin{center}\begin{tabular}{|c|c|c|c|c|c|c|c|c|c|}
  \hline
  mesh & $h_{\mesh}$ & \# cv &  \# vertices & \#  edges  \\
  \hline
 Mesh1\_1 &  $\frac 1 4$ &  56  &  37  &  92  \\
  \hline
 Mesh1\_2 & $\frac 1 {8}$ & 224  &  129  &  352  \\
  \hline
Mesh1\_3 &  $\frac 1 {16}$ &  896  &  481  &  1376  \\
  \hline
Mesh1\_4 &  $\frac 1 {32}$ &  3584  &  1857  &  5440  \\
  \hline
Mesh1\_5 &  $\frac 1 {64}$ &  14336  &  7297  &  21632  \\
  \hline
Mesh1\_6 &  $\frac 1 {128}$ &   57344  &  28929  &  86272  \\
  \hline
Mesh1\_7 &  $\frac 1 {256}$ & 229376  &  115201  &  344576  \\
  \hline
 \end{tabular} \end{center} 

 \medskip 
 
We show in Figure \ref{fig:resuirrell} a numerical solution obtained on  Mesh1\_3.
\begin{figure}
\resizebox{12cm}{!}{ \includegraphics{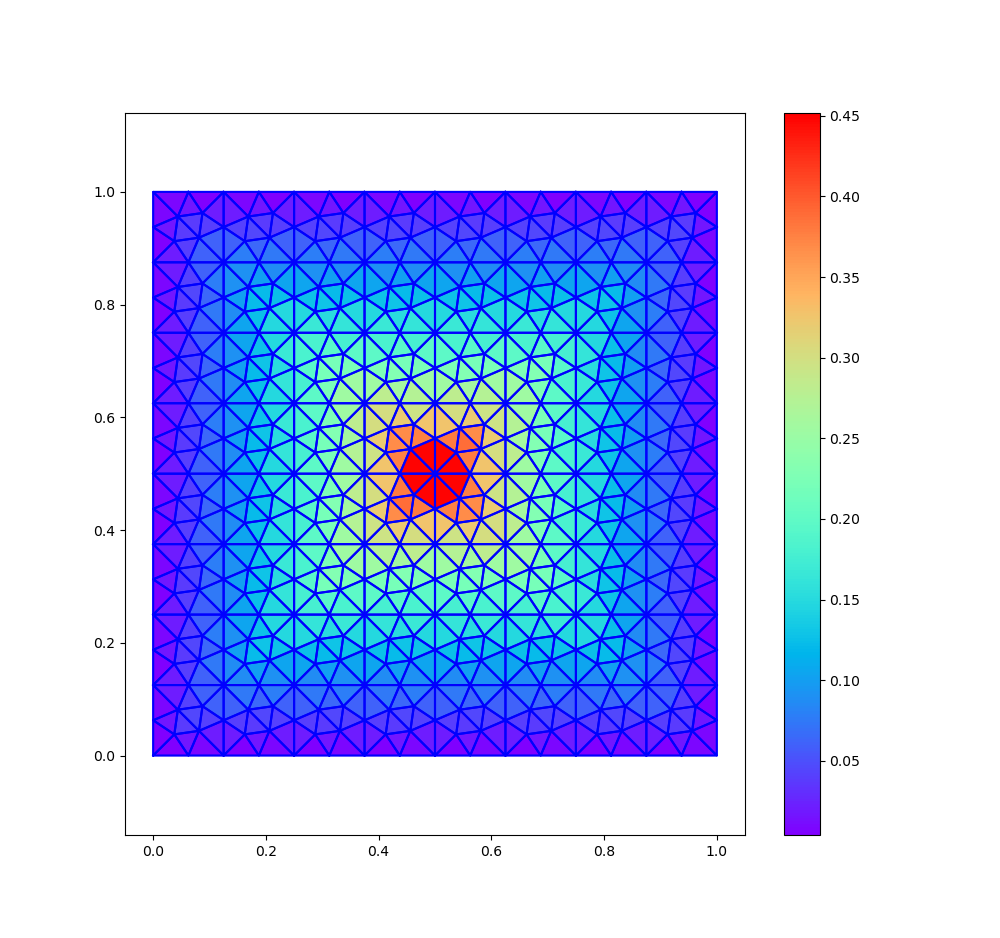}}
\caption{The approximate solution of Problem  \eqref{eq:pbellcont} computed on Mesh1\_3.}\label{fig:resuirrell}
\end{figure}

Note that these meshes share the same regularity factor $\theta_\mesh$ but do not feature any symmetries that might increase the order of convergence. 
 
For each mesh, we consider the interpolation $\overline{u}_{\mesh}\in \Xmesh$ of $\overline{u}$ defined as specified above. 
Figure \ref{fig:error} provides the approximation error and the interpolation error obtained on each mesh.
%

We observe that, as in the regular case, the error $\Vert \overline{u}_{\mesh} - u\Vert_{L^2}$ is much smaller than $\Vert \overline{u} - u\Vert_{L^2}$ (recall that superconvergence is observed on regular solutions), and that we observe an order 1 in the convergence of $\Vert \overline{u} - u \Vert_{L^2}$ to $0$. 
When accounting for the error on the gradient, we get that the approximation error behaves similarly for the small values of $h_{\mesh}$ as the interpolation error. Nevertheless, Figure \ref{fig:error} does not show that the quantities $\delta_{\mesh}(\overline{u},\overline{u}_{\mesh})$ and $\Vert \gnc\overline{u} - G_{\mesh}u\Vert_{L^2}$   behave as $(h_{\mesh})^\alpha$ for some $\alpha >0$. 
Our conjecture is that these terms behave as $(-\log( h_{\mesh})) ^{-\alpha}$ but the range of mesh sizes that can be tested is not sufficient to assess it.

\begin{center} 
\pgfplotstableread[col sep=&,row sep=\\]{
 h & e1 & e2 & e3 & e4 & e5 \\
 0.25&  1.116E-03 & 3.587E-02 & 2.193E-03 &  2.847E-02        &  7.867E-02 \\
0.125 &  5.383E-04 & 1.869E-02 & 2.256E-03 &   1.641E-02     & 7.207E-02 \\
0.0625 &  2.484E-04 & 9.624E-03 & 2.110E-03 &   9.789E-03    & 6.619E-02 \\
0.03125&   1.137E-04 & 4.920E-03 & 1.930E-03 &    6.208E-03  & 6.115E-02 \\
0.015625 &  5.219E-05 & 2.503E-03 & 1.764E-03 &   4.265E-03    & 5.687E-02 \\
0.0078125 &   2.409E-05 & 1.270E-03 & 1.621E-03 &   3.190E-03  & 5.317E-02 \\
0.00390625 &  1.119E-05 & 6.423E-04 & 1.499E-03 &   2.574E-03   & 4.999E-02 \\
}\mytable

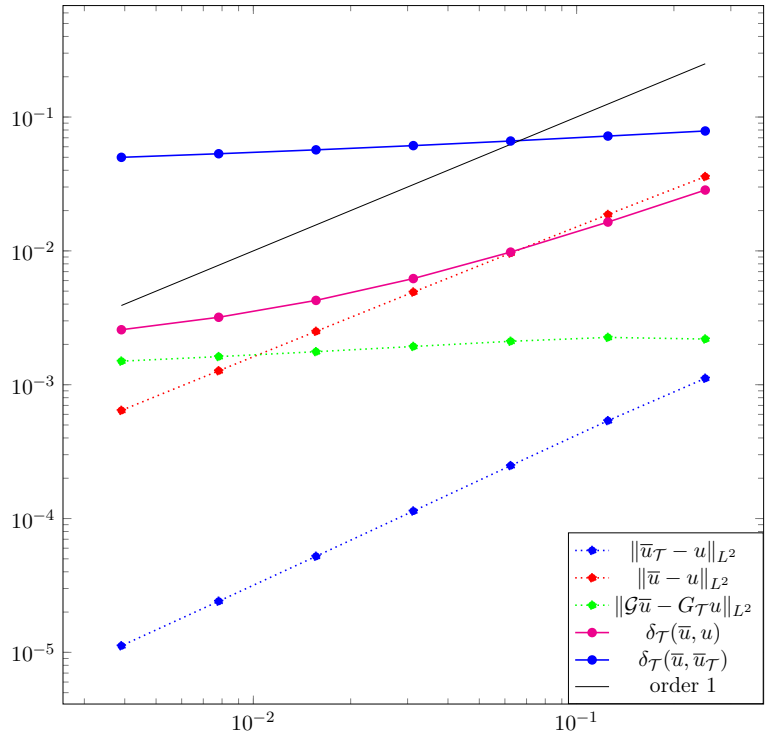
\begin{figure}[ht!]
 \resizebox{10cm}{!}{\begin{tikzpicture}
\begin{loglogaxis}[width=14cm, height=14cm,
legend style = {at={ (1,0)}, anchor = south east}]
        \addplot [color=blue, dotted, thick, mark=*
        ] table [x = h, y = e1] {\mytable};
        \addlegendentry{$\Vert \overline{u}_{\mesh} - u\Vert_{L^2}$}
        \addplot [color=red, dotted, thick, mark=*
        ] table [x = h, y = e2] {\mytable};
        \addlegendentry{$\Vert \overline{u} - u\Vert_{L^2}$}
        \addplot [color=green, dotted, thick, mark=*
        ] table [x = h, y = e3] {\mytable};
        \addlegendentry{$\Vert \gnc\overline{u} - G_{\mesh}u\Vert_{L^2}$ }
        \addplot [color=magenta, thick, mark=*
        ] table [x = h, y = e4] {\mytable};
        \addlegendentry{$\delta_{\mesh}(\overline{u},u)$}
        \addplot [color=blue, thick, mark=*
        ] table [x = h, y = e5] {\mytable};
        \addlegendentry{$\delta_{\mesh}(\overline{u},\overline{u}_{\mesh})$}
        \addplot [color=black,  thin, mark=none
        ] table [x = h, y = h] {\mytable};
        \addlegendentry{order 1}
    \end{loglogaxis}
\end{tikzpicture}}
\caption{Approximation and interpolation errors vs. the size of the mesh. } 
\label{fig:error}
\end{figure}
\end{center}

\section{Conclusion and perspectives}

In the present study, we were able to show an optimal error estimate for the discretization of a Poisson equation with minimal regularity, namely  when the solution possesses no greater regularity than belonging to $H_0^1 (\Omega)$, even though the two-point flux approximation (TPFA) cannot be placed in the framework of the gradient discretization method for which such an estimate already exists \cite{DEGGH}.  
Error bounds are derived for both the solution and the approximation of the gradient component orthogonal to the mesh faces. 
Our estimates are optimal, in the sense that the approximation error is shown to be of the same order  as the sum of the interpolation error and the conformity error. 
 Furthermore, numerical experiments show the validity of our error for a solution with minimal regularity. 
The results are extended to evolution problems discretized via the implicit Euler scheme, as detailed in the accompanying appendix.

While our study sheds some light on error estimates for finite volume schemes on admissible meshes for irregular and regular right-hand sides, the superconvergence of the method for regular right-hand sides remains an open problem; indeed, even though numerical experiments demonstrate that the two point scheme  superconverges on triangular \cite{BOIVIN2000806,herbin-labergerie} and Vorono\"{\i} meshes \cite{Mishev-voronoi, voronoi-uno, voronoi-due}, the theoretical proof remains an open problem except for the case of a (rather large) class of triangular meshes \cite{droniou-nataraj} or for a double finite volume scheme known as ``DDFV'' which involves two TPFA schemes, one on the Delaunay triangulation and one on the associated Vorono\"{\i} mesh \cite{omnes-voronoi}.

\appendix

\section{The transient case}\label{sec:heat}

This appendix is focused on the study of an optimal error bound in the transient case.   
For a given $T>0$, the solution of the continuous problem is $u : \Omega \times [0,T] \to \R$ such that
\begin{align}
 & \partial_t u- \Delta u = L \label{eq:laplace-parab}\\
 & u(\cdot,t) = 0 \mbox{ on }\partial \Omega \mbox{ for } t \in [0,T],  \label{eq:cldiric-parab} \\
 & u (0)-\Phi u(T)=\xi_0, \label{eq:ci-parab}
\end{align}
where $L \in L^2(0,T;H^{-1}(\Omega))$, $\xi_0\in L^2(\Omega)$ and $\Phi~:~L^2(\Omega)\to L^2(\Omega)$ is  a linear contractive mapping, {\it i.e.} such that  $\Vert\Phi v\Vert_{L^2(\Omega)}\le \Vert v\Vert_{L^2(\Omega)}$ for all $v\in L^2(\Omega)$. 
Note that if $\Phi = 0$, the problem \eqref{eq:laplace-parab}-\eqref{eq:ci-parab} is the standard Cauchy problem, and if $\Phi =$ Id, \eqref{eq:ci-parab} becomes a periodic condition.

Identifying $L^2(\Omega)$ with its dual space  thanks to the Riesz–Fr\'echet representation theorem, we have 
\[H^1_0(\Omega) \subset L^2(\Omega)=(L^2(\Omega))' \subset (H^1_0(\Omega))'=H^{-1}(\Omega),\]
so that the space 
\begin{equation*} 
W(0,T;\Omega) = \{u\in L^2(0,T;H_0^1(\Omega)): u'\in L^2(0,T;H^{-1}(\Omega))\},
\end{equation*}
is well defined, with $u'$ the weak derivative of the fonction $t \mapsto u(\cdot,t)$.
Recall that if $u \in W(0,T;\Omega)$, then $u \in C([0,T],L^2(\Omega))$.
As in the elliptic case, we decompose  $L = f + \div \bF$ with $f\in L^2(0,T;L^2(\Omega))$ and $\bF\in L^2(0,T;L^2(\Omega)^d)$.
A weak formulation of the problem \eqref{eq:laplace-parab}-\eqref{eq:ci-parab} then reads
\begin{align*} 
&\overline{u}\in W(0,T;\Omega)\hbox{ s.t.\ } \forall \overline{v}\in L^2(0,T;H_0^1(\Omega)), \\  
&\langle\overline{u}'(t), \overline{v}(t)\rangle_{H^{-1}(\Omega), H_0^1(\Omega)} + \langle\nabla \overline{u}(t) +\bF(t), \nabla\overline{v}(t)\rangle_{L^2(\Omega)^d} = \langle f(t) , \overline{v}(t)\rangle_{L^2(\Omega)}\hbox{ for a.e. }t\in [0,T]\\  
&\hbox{ and }\overline{u}(0)-\Phi \overline{u}(T)=\xi_0,
\end{align*}
which also reads
\begin{equation}\label{eq:pbcont}
\overline{u}\in W(0,T;\Omega)\hbox{ s.t.\ } \overline{u}' - {\rm div}( \nabla \overline{u} +\bF)=f\hbox{ and }\overline{u}(0)-\Phi \overline{u}(T)=\xi_0.
\end{equation}
Recall that $-\Delta$ is a linear continuous bijective operator from $H^1_0(\Omega)$ to $H^{-1}(\Omega)$ and therefore also from $L^2(0,T; H_0^1(\Omega))$ to $L^2(0,T;H^{-1}(\Omega))$.
We denote by $R$ its inverse which is linear continuous from $L^2(0,T; H^{-1}(\Omega))$ to $L^2(0,T;H_0^1(\Omega))$.
More precisely, noting that
\[
\forall (u, v)  \in H^1_0(\Omega)^2, \langle -\Delta u, v\rangle_{H^{-1}(\Omega), H_0^1(\Omega)}=
\langle \nabla u, \nabla v\rangle_{L^2(\Omega)^d}, 
\]
and that $u=R \xi\Longleftrightarrow  \xi=-\Delta u$, the operator $R$ is characterized by
\begin{equation}
\label{eq:defR}
 \forall (\xi,v)\in H^{-1}(\Omega)\times  H_0^1(\Omega),\ \langle \nabla R\xi , \nabla v\rangle_{L^2(\Omega)^d}= \langle \xi, v\rangle_{H^{-1}(\Omega), H_0^1(\Omega)}.
\end{equation}
The problem \eqref{eq:pbcont} is then equivalent to
\begin{equation}
\label{eq:pbcontriesz}
\hbox{find }\overline{u}\in W(0,T;\Omega)\hbox{ s.t.\ } -{\rm div}(\nabla R \overline {u}' +  \nabla \overline {u}  +\bF)=f\hbox{ and }\overline {u}(0)-\Phi \overline {u}(T)=\xi_0,
\end{equation}
which in particular implies that $\nabla R \overline {u}' +  \nabla \overline {u} + \bF\in L^2(0,T;\hdiv(\Omega))$.
Let us first state a known result on the well-posedness of the problem \eqref{eq:pbcont}. 
\begin{thm}[Existence and uniqueness of $\overline {u}$]\label{thm:5.3}
 For all  $f\in L^2(0,T; L^2(\Omega))$, $\bF\in L^2(0,T;L^2(\Omega)^d)$ and $\xi_0\in L^2(\Omega)$, Problem \eqref{eq:pbcont} has a unique solution.
\end{thm}
The proof ot this theorem may be found in \cite{Lio61} in the case that $\Phi = 0$ and in \cite{RDV1} for any general contractive mapping $\Phi$.
\subsection{Description of the implicit Euler scheme}\label{sub:4.1}
Let $N\in\mathbb{N}\setminus\{0\}$ and define the time step (taken to be uniform for simplicity of presentation) $k = \frac {T}{N}$. 
 For all $m\in \llbracket 1,N \rrbracket$,  $f^{(m)}\in  L^2(\Omega)$, $\bF^{(m)}\in  L^2(\Omega)^d$ are defined by
\[
 f^{(m)} = \frac 1 {k}\int_{(m-1)k}^{m k} f(t) {\rm d}t\quad\hbox{ and }\quad\bF^{(m)} = \frac 1 {k}\int_{(m-1)k}^{m k} \bF(t) {\rm d}t.
\]
The implicit Euler scheme consists in seeking $N+1$ elements of the space $\Xmesh$ given in Definition \ref{def:meshdirichlet}, denoted by $(u^{(m)})_{m=0,\ldots,N}$, such that
\begin{subequations}
\label{eq:scheme}
\begin{align}
&\mbox{\emph{Initialization:}} \; \langle  u^{(0)} - \Phi   u^{(N)} ,   v\rangle_{ L^2(\Omega)}  =  \langle \xi_0 ,   v\rangle_{ L^2(\Omega)}\mbox{ for all }v\in \Xmesh
\label{eq:schemeiniex}
\\
&\mbox{\emph{Step} } m: \;\qquad \langle \frac{u^{(m)}-u^{(m-1)}}{k},  v\rangle_{L^2(\Omega)} + d \langle G_{\mesh} u^{(m)},G_{\mesh}  v\rangle_{ L^2(\Omega)} =\langle f^{(m)}, v\rangle_{ L^2(\Omega)} \nonumber
\\  &   \hspace{4.5cm} -\langle \bF^{(m)},\nabla_{\mesh} v \rangle_{ L^2(\Omega)^d}, \quad\mbox{ for all }m\in \llbracket 1,N \rrbracket \mbox{ and }v\in \Xmesh. \label{eq:schemeimp} 
\end{align}
\end{subequations}
Note that, if $\Phi\equiv 0$, the scheme is the usual implicit scheme, and the existence and uniqueness of a solution $( u^{(0)},(u^{(m)})_{m\in \llbracket 1,N \rrbracket})$ to \eqref{eq:scheme} is standard. 
In the general case, a linear system involving $u^{(0)}$ must be solved, and its invertibility is proved by Theorem \ref{thm:errest}.
Let us now introduce the discrete functional space $W_\mesh$, which consists of all functions $w~:~[0,T]\to \Xmesh$ in the sense that there exist $N+1$ elements of $\Xmesh$, denoted by $(w^{(m)})_{m=0,\ldots,N}$, such that
 \begin{equation}\label{eq:schememtheta}
 \begin{aligned}
   w(0)={}&w^{(0)},\mbox{ and }\\ 
   w(t) ={}&  w^{(m)}\hbox{ for all }t\in ((m-1)k, mk],\hbox{ for all }m\in \llbracket 1,N \rrbracket. 
 \end{aligned}
 \end{equation}
Note that such a function $w\in W_\mesh$ is a function of time with values in a (discrete) functional space; however, observe that  there is a bijection from the space $W_\mesh$ to the space $\Xmesh^{N+1}$, through the mapping $w\mapsto (w(mk))_{m\in \llbracket 0,N \rrbracket}$, so that $w$ may also be seen as a function of space and time, defined, with an abuse of notation, by
$w: (x,t) \in \Omega \times[0,T]  \mapsto w(x,t) = w^{(0)}(x) $ if $t=0$ and $w(x,t) = w^{(m)}(x)$ if $t\in ((m-1)k, mk]$. 
Therefore, to avoid any confusion, we denote the  discrete time derivative of $w$ by $\eth_t  w$, defined as follows:
 \begin{equation}\label{eq:defqnevolstrex}
   \eth_t w(t) = \frac{w^{(m)}-w^{(m-1)}}{k},\qquad\text{ for a.e. } t\in ((m-1)k, mk),\hbox{ for all }m\in\llbracket 1,N \rrbracket,
 \end{equation}
and its discrete gradient $G_{\mesh} v : [0,T] \to X_\mesh$ by 
\[
 G_{\mesh} w (x,t) = \begin{cases}
                G_{\mesh} w^{(0)}(x) \mbox{ if } t=0, \\
                G_{\mesh} w^{(m)}(x) \mbox{ if } t\in ((m-1)k, mk],\hbox{ for } m\in \llbracket 1,N \rrbracket,
               \end{cases}
\]
where $G_{\mesh}:  \Xmesh\to L^2(\Omega)$ be the discrete normal gradient defined by \eqref{eq:defgm}.
 We also need to define the space $V_\mesh$ of all functions $v\in L^2(0,T; \Xmesh)$ for which there exist $N$ elements of $\Xmesh$, denoted by $(v^{(m)})_{m\in \llbracket 1,N \rrbracket}$, such that
 \begin{equation}\label{eq:defVdisc}
   v(t) =  v^{(m)}\hbox{ for a.e. }t\in ((m-1)k, mk),\hbox{ for all }m\in \llbracket 1,N \rrbracket. 
 \end{equation}
 Observe that there is now a bijection from the space $V_\mesh$ to $\Xmesh^{N}$, through the mapping $w\mapsto (w(mk))_{m\in \llbracket 1,N \rrbracket}$.
 \begin{remark}[Difference between $W_\mesh$ and $V_\mesh$]
The functions belonging to $W_\mesh$ are defined \emph{pointwise and everywhere}, whereas those belonging to $V_\mesh$ are only defined \emph{almost everywhere} on $(0,T)$.
  \end{remark}
The scheme \eqref{eq:schemeiniex}--\eqref{eq:schemeimp} can then be written under the following form (for the sake of brevity, here and in the sequel, we omit $0,T$ in the expression of some spaces of time dependent functions: for example, we write $L^2( L^2(\Omega))$ instead of $L^2(0,T; L^2(\Omega))$): 
\begin{subequations}
\label{eq:schemevar}
 \begin{align}
 \label{eq:schemeformeb}
&\hbox{Find }u\in W_\mesh\hbox{ such that }\forall (v,z)\in V_\mesh\times \Xmesh, \quad b(u, (v,z)) = \mathcal L((v,z)), \mbox{ with }\\
& b(u, (v,z)) = \langle  \eth_t u,  v\rangle_{ L^2( L^2(\Omega)) } + d \langle G_{\mesh} u,G_{\mesh}  v\rangle_{L^2( L^2(\Omega))}+\langle  u(0) - \Phi   u(T) ,   z\rangle_{ L^2(\Omega)}  ,\label{eq:defb}\\
&\mathcal L((v,z)) =  \langle f,  v \rangle_{L^2( L^2(\Omega))} -\langle \bF,\nabla_{\mesh}  v  \rangle_{L^2( L^2(\Omega)^d)}+  \langle \xi_0 ,   z\rangle_{ L^2(\Omega)}. \nonumber
 \end{align}
\end{subequations}

\subsection{An optimal error estimate}
\label{sec:error_estimate}

Let $R_\mesh: v \in \Xmesh\mapsto R_\mesh v\in\Xmesh$, with $R_\mesh v$ defined by 
\begin{equation}\label{eq:defRdisc}
d \langle G_{\mesh} R_\mesh v, G_{\mesh} w\rangle_{ L^2(\Omega)^d} = \langle  v,   w\rangle_{ L^2(\Omega)}\quad\mbox{ for all }w\in \Xmesh.
\end{equation}
With this definition, the scheme \eqref{eq:schemevar} can be recast as: for all $(v,z)\in V_\mesh\times \Xmesh$,
\begin{multline}\label{eq:schemevar.recast}
 d\langle G_{\mesh} R_\mesh\eth_t u+ G_{\mesh} u, G_{\mesh} v\rangle_{ L^2( L^2(\Omega)^d) } +\langle \bF,\nabla_{\mesh}  v  \rangle_{L^2( L^2(\Omega)^d)}
\\ +\langle  u(0) - \Phi   u(T) ,   z\rangle_{ L^2(\Omega)}
=
\langle f,  v \rangle_{L^2( L^2(\Omega))} +  \langle \xi_0 ,   z\rangle_{ L^2(\Omega)}.
\end{multline}
We define a distance between any functions $\varphi\in W(0,T;\Omega)$ and  $v\in W_\mesh$ by
\begin{equation}\label{eq:def.Suv}
\begin{aligned}
 \delta^{(T)}_\mesh(\varphi,v) ={}& d\Vert\gnc_{\mesh}R \varphi' - G_{\mesh} R_\mesh\eth_t v\Vert_{L^2( L^2(\Omega))}\\
 &+d\Vert \gnc_{\mesh} \varphi - G_{\mesh}  v\Vert_{L^2( L^2(\Omega))}+ \max_{t\in[0,T]}\Vert \varphi(t) -   v(t)\Vert_{ L^2(\Omega)}.
\end{aligned}
\end{equation}
(Recall that $\gnc_{\mesh}$ is the mean normal gradient defined by \eqref{eq:defogm}.)
The interpolation error for a given function $\varphi \in W(0,T;\Omega)$ is then defined by
\begin{equation}\label{eqdef:interp-time}
 \mathcal I^{(T)}_\mesh(\varphi) =   \inf_{v\in W_\mesh} \delta^{(T)}_\mesh(\varphi,v)
\end{equation}
We also define the time-space conformity error
 $\zeta^{(T)}_{\mesh}:L^2(0,T;\hdiv(\Omega)) \to [0,+\infty)$  by
\be
\forall\bvarphi\in L^2(0,T;\hdiv(\Omega)),\ \zeta^{(T)}_{\mesh}(\bvarphi) =  \sup_{w\in V_{\mesh}\setminus\{0\}}\frac{\dsp \left|\langle \bvarphi,\nabla_{\mesh} w\rangle_{L^2( L^2(\Omega)^d)} + \langle {\rm div}\bvarphi, w\rangle_{L^2( L^2(\Omega))} \right|}{\Vert G_{\mesh} w \Vert_{L^2( L^2(\Omega)^d)}}.
\label{abs:defwdiscV}
\ee
Before proving the optimal error result, we need the following preliminary lemma.
\begin{lem}\label{lem:hypsufbnb}
For $w\in \Xmesh$, the following inequalities hold
\begin{multline}\label{eq:majunif}
 \max_{t\in[0,T]} \Vert w(t)\Vert_{L^2(\Omega)} 
 \le  \Vert d\, G_{\mesh} R_\mesh\eth_t w\Vert_{L^2( L^2(\Omega))}+\Vert d\, G_{\mesh}  w\Vert_{L^2( L^2(\Omega))}+ \Vert w(0)\Vert_{ L^2(\Omega)},
\end{multline} 
\begin{equation}\label{eq:hypsufbnb.1}
\langle  G_{\mesh} R_\mesh\eth_t w,  d\, G_{\mesh} w\rangle_{L^2( L^2(\Omega)) }
\ge \frac12 \Vert   w(T)\Vert_{L^2(\Omega)}^2 - \frac12\Vert  w(0)\Vert_{L^2(\Omega)}^2,
\end{equation}
and, recalling that ${\rm diam}(\Omega)$ is defined by \eqref{eq:poindis},
\begin{equation}\label{eq:hypsufbnb.2}
\Vert  d\, G_{\mesh} R_\mesh\eth_t w\Vert_{L^2( L^2(\Omega))}^2 + \left(1+\frac{{\rm diam}(\Omega)^2}{T}\right) \Vert d\, G_{\mesh} w\Vert_{L^2( L^2(\Omega))}^2
\ge \Vert   w(T)\Vert_{L^2(\Omega)}^2.
\end{equation}
\end{lem}

\begin{proof}
Let $w\in \Xmesh$; thanks to the equality $(a-b)a=\frac12 a^2+\frac12 (a-b)^2-\frac12 b^2$, the definition \eqref{eq:defRdisc} of $R_\mesh$ yields, for $0\le m \le n\le N$,
\begin{align}
 \int_{mk}^{nk} \langle  G_{\mesh} R_\mesh\eth_t w(t),{}&  d\, G_{\mesh} w(t)\rangle_{ L^2(\Omega)}{\rm d}t =
 \int_{mk}^{nk} \langle \eth_t w(t),   w(t)\rangle_{ L^2(\Omega) }{\rm d}t \nonumber\\
 ={}& \sum_{p=m}^{n-1} k \langle  \frac {w^{(p+1)} - w^{(p)}} k,  w^{(p+1)}\rangle_{L^2(\Omega)} \nonumber\\
 ={}& \frac 1 2 \Vert   w^{(n)}\Vert_{L^2(\Omega)}^2 +\frac 1 2\sum_{p=m}^{n-1} \Vert   (w^{(p+1)} - w^{(p)})\Vert_{L^2(\Omega)}^2 -\frac 1 2 \Vert   w^{(m)}\Vert_{L^2(\Omega)}^2.
\label{eq:ineqestimthetaim}
\end{align}
Applying the Cauchy--Schwarz inequality to the left-hand side, we get
\begin{align}
 \frac 1 2 {}&\Vert  w^{(n)}\Vert_{L^2(\Omega)}^2  \le{}   \Vert  d\, G_{\mesh} R_\mesh\eth_t w\Vert_{L^2( L^2(\Omega)^d)}\Vert d\, G_{\mesh} w\Vert_{L^2( L^2(\Omega)^d)}+ \frac 1 2 \Vert  w^{(m)}\Vert_{L^2(\Omega)}^2\nonumber\\
 &\le\frac12\Vert  d\, G_{\mesh} R_\mesh\eth_t w\Vert_{L^2( L^2(\Omega)^d)}^2+\frac12\Vert d\, G_{\mesh} w\Vert_{L^2( L^2(\Omega)^d)}^2+ \frac 1 2 \Vert  w^{(m)}\Vert_{L^2(\Omega)}^2,
 \label{eq:est.m.mp}
 \end{align}
where the second line follows from the Young inequality. 
Setting $m=0$ allows us to take any $n=0,\ldots,N$. 
Taking the square root of the above inequality and using $(a^2+b^2+c^2)^{\nicefrac12}\le a+b+c$ then concludes the proof of \eqref{eq:majunif}.

The inequality \eqref{eq:hypsufbnb.1} is obtained letting $m=0$ and $n=N$ in \eqref{eq:ineqestimthetaim}. To prove \eqref{eq:hypsufbnb.2}, we come back to \eqref{eq:est.m.mp} and set $n=N$ to get, after multiplication by $2k$, for all $m=0,\ldots,N$,
\[
 k \Vert   w(T)\Vert_{L^2(\Omega)}^2  \le   k   \Vert  d\, G_{\mesh} R_\mesh\eth_t w\Vert_{L^2( L^2(\Omega)^d)}^2+k\Vert d\, G_{\mesh} w\Vert_{L^2( L^2(\Omega)^d)}^2+ k \Vert  w^{(m)}\Vert_{L^2(\Omega)}^2.
\]
Summing over $m\in \llbracket 1,N \rrbracket$ yields
\begin{align*}
  T \Vert  w(T)\Vert_{L^2(\Omega)}^2  \le{}& T  \Vert  d\, G_{\mesh} R_\mesh\eth_t w\Vert_{L^2( L^2(\Omega)^d)}^2+T\Vert d\, G_{\mesh} w\Vert_{L^2( L^2(\Omega)^d)}^2 + \Vert  w\Vert_{L^2( L^2(\Omega))}^2\\
  \le{}&T  \Vert  d\, G_{\mesh} R_\mesh\eth_t w\Vert_{L^2( L^2(\Omega)^d)}^2+(T+{\rm diam}(\Omega)^2)\Vert d\, G_{\mesh} w\Vert_{L^2( L^2(\Omega)^d)}^2,
\end{align*}
which proves \eqref{eq:hypsufbnb.2}.
\end{proof}

We now turn to the error estimate result.

\begin{thm}\label{thm:erresttime}
 There exists one and only one solution $u$ to \eqref{eq:scheme}. 
 Moreover, letting $\overline {u}$ be the solution to \eqref{eq:pbcont} and $\bv$ be defined by
\begin{equation}\label{eq:defbv}
 \bv := \nabla R \overline {u}' +  \nabla \overline {u}  +\bF\in L^2(0,T;\hdiv(\Omega)),
\end{equation}
 there exists $C_{\Omega,T}\ge 0$, depending only on $\Omega$ and $T$, such that:
\begin{multline}\label{eq:errestgd}
\frac 1 2\Big[ \zeta^{(T)}_\mesh(\bv)+ \mathcal I^{(T)}_\mesh(\overline {u},v)\Big] \le \delta^{(T)}_\mesh(\overline {u},u) 
\le  C_{\Omega,T} \Big[ \zeta^{(T)}_\mesh(\bv) + \mathcal I^{(T)}_\mesh(\overline {u},v)\Big].
\end{multline}
\end{thm}

\begin{remark}[Optimality of the error estimate \eqref{eq:errestgd}]
As in the steady case, the second inequality in \eqref{eq:errestgd} gives an error estimate for the scheme, while the first inequality shows its optimality. 
\end{remark}

\begin{proof}[Proof of Theorem \ref{thm:erresttime}.]
Let $v\in V_{\mesh}$ and $z\in \Xmesh$ be given. Definition \eqref{abs:defwdiscV} of $\zeta^{(T)}_{\mesh}(\bv)$ gives
\[
 \int_0^T \Big(\langle \bv(t),\nabla_{\mesh}  v (t)\rangle_{ L^2(\Omega)^d} + \langle {\rm div}\bv(t), v(t)\rangle_{ L^2(\Omega)}
\Big){\rm d}t
\le  \zeta^{(T)}_{\mesh}(\bv) \sqrt{d}\Vert G_{\mesh} v\Vert_{L^2( L^2(\Omega))}.
\]
This yields, using \eqref{eq:pbcontriesz} (which reads ${\rm div} \bv=-f$), and \eqref{eq:schemevar.recast},
\begin{multline*}
\int_0^T d\Big(\langle  \gnc_{\mesh} R \overline {u}'(t) + \gnc_{\mesh} \overline {u}(t)  -  ( G_{\mesh} R_{\mesh}\eth_t u(t) +  G_{\mesh} u(t)) ,G_{\mesh} v(t)\rangle_{ L^2(\Omega)} \Big){\rm d}t\\
+\langle\xi_0 - ( u(0) - \Phi   u(T)) ,   z\rangle_{ L^2(\Omega)}
\le  \zeta^{(T)}_{\mesh}(\bv) \sqrt{d}\Vert G_{\mesh} v\Vert_{L^2( L^2(\Omega))}.
\end{multline*}
Using $\overline {u}(0) - \Phi \overline {u}(T) = \xi_0$, we get
\begin{multline}
\int_0^T d\Big(\langle  \gnc_{\mesh} R \overline {u}'(t) + \gnc_{\mesh}  \overline {u}(t)  - ( G_{\mesh} R_{\mesh}\eth_t u(t) +  G_{\mesh} u(t)) ,G_{\mesh} v(t)\rangle_{ L^2(\Omega)} 
\Big){\rm d}t
\\
 +\langle \overline {u}(0) - \Phi \overline {u}(T) - ( u(0) - \Phi   u(T)) ,   z\rangle_{ L^2(\Omega)}
\\
\le  \zeta^{(T)}_{\mesh}(\bv) \sqrt{d}\Vert G_{\mesh} v\Vert_{L^2( L^2(\Omega))}.
\label{eq:error.est.1}
\end{multline}
We then take an arbitrary element $\widetilde{v}\in W_{\mesh}$ and notice that, by definition \eqref{eq:def.Suv} of $\delta^{(T)}_{\mesh}(\overline {u},\widetilde{v})$ and since $\Phi$ is a contraction, 
\begin{multline*}
\int_0^T d \langle [G_{\mesh} R_{\mesh}\partial \widetilde{v}-\gnc_{\mesh} R \overline {u}'](t),G_{\mesh} v(t)\rangle_{ L^2(\Omega)} + \langle [G_{\mesh} \widetilde{v}-\gnc_{\mesh} \overline {u}](t) ,G_{\mesh} v(t)\rangle_{ L^2(\Omega)} 
{\rm d}t  \\
+\langle \widetilde{v}(0)-\overline {u}(0) - \Phi (  \widetilde{v}(T)-\overline {u}(T)) ,   z\rangle_{ L^2(\Omega)} \\
\le \delta^{(T)}_{\mesh}(\overline {u},\widetilde{v})\Vert G_{\mesh} v\Vert_{L^2( L^2(\Omega))} +   2\delta^{(T)}_{\mesh}(\overline {u},\widetilde{v})\Vert  z\Vert_{ L^2(\Omega)}.
\end{multline*}
Adding this inequality to \eqref{eq:error.est.1} yields
\begin{multline}\label{eq:transient1}
\int_0^T d\langle G_{\mesh} R_{\mesh}\partial (\widetilde{v} - u)(t) + G_{\mesh} (\widetilde{v}(t) - u(t)) ,G_{\mesh} v(t)\rangle_{ L^2(\Omega)} 
{\rm d}t  \\
\langle (\widetilde{v}(0)  -  u(0)) - \Phi   (\widetilde{v}(T) -u(T)) ,   z\rangle_{ L^2(\Omega)} \\
\le    
(\delta^{(T)}_{\mesh}(\overline {u},\widetilde{v})+\zeta^{(T)}_{\mesh}(\bv))\Vert G_{\mesh} v\Vert_{L^2( L^2(\Omega)^d)} +   2\delta^{(T)}_{\mesh}(\overline {u},\widetilde{v})\Vert  z\Vert_{ L^2(\Omega)}.
\end{multline}

Let us now introduce the notations used in Lemma \ref{lem:suffbnb}, and prove that the hypotheses of the lemma hold. 
We denote by $V=L^2(0,T; L^2(\Omega))$ the Hilbert space endowed with the inner product $d\langle{\cdot},{\cdot}\rangle_{L^2( L^2(\Omega))}$, and by $L= L^2(\Omega)$. Let $Z$ and $Y$ be the Hilbert spaces defined by $Z = {V}\times {V}\times  L \times  L $ and $Y = {V}\times  L $. 
We define the bilinear form $\widehat{b}:Z\times Y\to \mathbb{R}$ by
\begin{equation}\label{eq:defblemmevf} 
 \widehat{b}((z_1,z_2,z_3,z_4),(y_1,y_2)) = d\langle z_1 + A z_2,y_1\rangle_{{V}} + \langle z_3 - \Phi z_4,y_2\rangle_{ L },
 \end{equation}
where $A={\rm Id}$.  
We then define the subspace $X\subset Z$ by
\[
X=G_{\mesh}(V_{\mesh})\times G_{\mesh}(V_{\mesh})\times \Xmesh\times \Xmesh.
\]
The operator $A={\rm Id}$ satisfies \eqref{eq:Malpha.operator} with $\alpha = M = 1$. Condition \eqref{eq:hypxunxdeux} is satisfied since $X_1 = X_2 = G_{\mesh}(V_{\mesh})$.
Let us prove that we can find $\zeta>0$ and $\delta >0$ such that \eqref{eq:condlim} holds.
Applying Lemma \ref{lem:hypsufbnb}, we add \eqref{eq:hypsufbnb.1} to $(1+\frac{{\rm diam}(\Omega)^2}{T})^{-1}\frac{1}{12}\times\eqref{eq:hypsufbnb.2}$. This shows that the hypothesis \eqref{eq:condlim} is satisfied with
\[
\zeta := \mu=\frac12+\left(1+\frac{{\rm diam}(\Omega)^2}{T}\right)^{-1}\frac{1}{12}\quad\mbox{ and }\quad
\nu=\frac12.
\]
We note that $\mu-\nu\|\Phi\|^2\ge \mu-\nu=\left(1+\frac{{\rm diam}(\Omega)^2}{T}\right)^{-1}\frac{1}{12}=:\delta$.

 We can therefore apply Lemma \ref{lem:suffbnb}, which yields the existence of $\widehat{\beta}>0$ depending only on ${\rm diam}(\Omega)$ and $T$ such that
 \begin{multline}\label{eq:resultsuffbnb}
  \sup_{(y_1,y_2)\in X_2\times X_3, \Vert (y_1,y_2)\Vert_Y = 1} \widehat{b}((z_1,z_2,z_3,z_4),(y_1,y_2))
  \ge  \widehat{\beta} \Vert (z_1,z_2,z_3,z_4)\Vert_{Z} \\
  \forall (z_1,z_2,z_3,z_4)\in X.
 \end{multline}
 We now remark that Inequality \eqref{eq:transient1} can be expressed by
\begin{equation}\label{eq:est.key.1}
 \widehat{b}((z_1,z_2,z_3,z_4),(y_1,y_2)) \le \widehat{c}_1 \Vert y_1\Vert_{V} + \widehat{c}_2 \Vert y_2\Vert_{ L^2(\Omega)},
\end{equation}
with
\begin{alignat*}{3}
z_1 ={}& G_{\mesh} R_{\mesh}\partial (\widetilde{v} - u), &\qquad z_2 ={}& G_{\mesh} (\widetilde{v} - u) , \\
z_3 ={}& (\widetilde{v}(0)  -  u(0)), &\qquad z_4 ={}&   (\widetilde{v}(T) -u(T)),\\
y_1 ={}& G_{\mesh} v,&\qquad y_2 ={}&  z,\\
\widehat{c}_1 ={}& \delta^{(T)}_{\mesh}(\overline {u},\widetilde{v})+\zeta^{(T)}_{\mesh}(\bv),&\qquad
\widehat{c}_2 ={}& 2\delta^{(T)}_{\mesh}(\overline {u},\widetilde{v}).
\end{alignat*}

Taking in the right-hand-side of \eqref{eq:est.key.1} the maximum over all $(y_1,y_2)\in G_{\mesh}(V_{\mesh})\times \Xmesh$ with norm in $V\times  L^2(\Omega)$ equal to 1  and using \eqref{eq:resultsuffbnb}, we deduce that
\begin{multline}\label{eq:infsup}
 \widehat{\beta}\Big( \Vert G_{\mesh} R_{\mesh}\partial (\widetilde{v}  -  u)\Vert_{L^2( L^2(\Omega)^d)}^2 +\Vert G_{\mesh}  (\widetilde{v}  -  u)\Vert_{L^2( L^2(\Omega)^d)}^2\\ 
 + \Vert (\widetilde{v}  -  u)(0)\Vert_{L^2(\Omega)}^2+ \Vert (\widetilde{v}  -  u)(T)\Vert_{L^2(\Omega)}^2\Big)^{\nicefrac12}\\
 \le \left[ \left(\delta^{(T)}_{\mesh}(\overline {u},\widetilde{v})+\zeta^{(T)}_{\mesh}(\bv) \right)^2 +  4\delta^{(T)}_{\mesh}(\overline {u},\widetilde{v})^2\right]^{\nicefrac12}\\
 \le \delta^{(T)}_{\mesh}(\overline {u},\widetilde{v})+\zeta^{(T)}_{\mesh}(\bv)  +  2\delta^{(T)}_{\mesh}(\overline {u},\widetilde{v}),
\end{multline}
where we use $(a^2+b^2)^{\nicefrac12}\le a+b$ for positive $a$ and $b$ in the last inequality.
Plugging this into \eqref{eq:infsup} and using \eqref{eq:majunif} in Lemma \ref{lem:hypsufbnb} together with $a+b+c\le \sqrt{3}(a^2+b^2+c^2)^{\nicefrac12}$, we infer
\begin{multline*}
 \widehat{\beta}\Big( \Vert G_{\mesh} R_{\mesh}\partial (\widetilde{v}  -  u)\Vert_{L^2( L^2(\Omega)^d)} +\Vert G_{\mesh}  (\widetilde{v}  -  u)\Vert_{L^2( L^2(\Omega)^d)}
  + \max_{t\in[0,T]} \Vert (\widetilde{v}  -  u)(t)\Vert_{L^2(\Omega)}\Big)\\
 \le \sqrt{3}\left( 3\delta^{(T)}_{\mesh}(\overline {u},\widetilde{v})+\zeta^{(T)}_{\mesh}(\bv)\right).
\end{multline*}
Using the triangle inequality in the definition \eqref{eq:def.Suv} of $\delta^{(T)}_{\mesh}$, we infer
\[
 \widehat{\beta} \delta^{(T)}_{\mesh}(\overline {u}, u)
 \le \sqrt{3}\left( 3\delta^{(T)}_{\mesh}(\overline {u},\widetilde{v})+\zeta^{(T)}_{\mesh}(\bv)\right)+ \widehat{\beta}\delta^{(T)}_{\mesh}(\overline {u}, \widetilde{v}).
\]
Since $\widetilde{v}$ is arbitrary in $W_{\mesh}$, this concludes the proof of the second inequality in \eqref{eq:errestgd}.

\medskip

Let us now turn to the first inequality in \eqref{eq:errestgd}; first note that
\[
 \inf_{v\in W_{\mesh}} \delta^{(T)}_{\mesh}(\overline {u},v) \le \delta^{(T)}_{\mesh}(\overline {u},u).
\]
To bound $\zeta^{(T)}_{\mesh}(\bv)$ we recall that $\bv=\nabla R \overline {u}' +  \nabla \overline {u}  +\bF$ satisfies $-{\rm div} \bv=f$ (see \eqref{eq:pbcontriesz}), and use the scheme \eqref{eq:schemevar.recast} (with $z=0$) to write, for any $v\in V_{\mesh}\backslash\{0\}$,
\begin{align*}
\langle \bv,G_{\mesh} v\rangle_{L^2(\Omega)^d}+{}&\langle{\rm div}\bv, v\rangle_{L^2(\Omega)}\\
={}& \int_0^T \Big(\langle  \nabla R \overline {u}'(t) + \nabla  \overline {u}(t)  \\
 &\qquad\qquad- ( G_{\mesh} R_{\mesh}\eth_t u(t) +  G_{\mesh} u(t)) ,G_{\mesh} v(t)\rangle_{ L^2(\Omega)^d} 
\Big){\rm d}t \\
\le{}& \left(\Vert (\nabla R \overline {u}' - G_{\mesh} R_{\mesh}\eth_t u)\Vert_{L^2( L^2(\Omega)^d)} + \Vert(\nabla  \overline {u} - G_{\mesh} u)\Vert_{L^2( L^2(\Omega)^d)}\right)\\
&\times\Vert G_{\mesh} v\Vert_{L^2( L^2(\Omega)^d)}.
\end{align*}
Dividing by $\Vert G_{\mesh} v\Vert_{L^2( L^2(\Omega)^d)}$ and taking the supremum over $v\in V_{\mesh}\backslash\{0\}$ shows that $\zeta^{(T)}_{\mesh}(\bv)\le \delta^{(T)}_{\mesh}(\overline {u},u)$, which concludes the proof.

\medskip

Finally, as in the steady case, we notice that $( u^{(0)},(u^{(m)})_{m\in \llbracket 1,N \rrbracket})$ is solution to a square linear system. Therefore the error estimate Theorem \ref{thm:errest} proves that, for a null right-hand-side, the solution is null, which shows that the system is invertible.

\end{proof}

The following lemma is proven in \cite{DEGGH}.

\begin{lem}[\protect{\cite[Lemma 4.7]{DEGGH}}]\label{lem:suffbnb}
Let ${V}$ and $ L $ be Hilbert spaces. Let $Z$ and $Y$ be the Hilbert spaces defined by $Z = {V}\times {V}\times  L \times  L $ and $Y = {V}\times  L $. Let $A:{V}\to {V}$ be an $M$-continuous and $\alpha$-coercive linear operator  (with $M\ge 1$ and $\alpha>0$), which means that 
\begin{equation}
\label{eq:Malpha.operator}
\Vert A v\Vert_{{V}} \le M \|v\|_{V}\quad\mbox{and}\quad
\langle A v, v\rangle_{{V}} \ge \alpha \|v\|_{V}^2\quad\forall v\in {V}.
\end{equation} 
Let $\Phi: L \to  L $ be a linear operator such that $\Vert \Phi\Vert\le 1$, and let $\widehat{b}~:~Z\times Y\to \mathbb{R}$ be defined by
 \begin{equation}\label{eq:defblemme} 
 \widehat{b}((z_1,z_2,z_3,z_4),(y_1,y_2)) = \langle z_1 + A z_2,y_1\rangle_{{V}} + \langle z_3 - \Phi z_4,y_2\rangle_{ L },
 \end{equation}
for all $(z_1,z_2,z_3,z_4)\in Z$ and for all $(y_1,y_2)\in Y$.

 Let $X\subset Z$ be a subspace of $Z$. We define the Hilbert spaces $X_1\subset V$, $X_2\subset V$, $X_3\subset  L $ and $X_4\subset  L $ by: for $i=1,\ldots,4$,
 \[
  X_i = \overline{\{ x_i: \ x\in X\} },\hbox{ where }x_i\hbox{ is the $i$-th component of }x\in Z.
 \]
 Assume that 
 \begin{equation}\label{eq:hypxunxdeux}
 X_1\subset X_2,
 \end{equation}
 and that there exist $\zeta>0$ and $\delta >0$ such that, for all $x\in X$,
\begin{equation}\label{eq:condlim}
 \langle x_1,x_2\rangle_{{V}} +\frac {\alpha^2} {12\, M^3}(\Vert  x_2\Vert_{V}^2 + \Vert x_1\Vert_{{V}}^2)\ge \mu\Vert x_4\Vert_{ L }^2 -\nu \Vert x_3\Vert_{ L }^2,
\end{equation}
for some $\mu\in (0,\zeta]$ and $\nu\in [0,\mu]$ with $\mu - \nu\Vert\Phi\Vert^2 \ge \delta$.
 Then, there exists $\widehat{\beta}>0$, only depending on $\alpha$, $M$, $\zeta$ and $\delta$ (and not on $\mu$, $\nu$ and $\Vert\Phi\Vert$) such that
 \begin{equation}\label{eq:bnbdisc}
  \sup_{y\in X_2\times X_3, \Vert y\Vert_Y = 1} \widehat{b}(x,y) \ge  \widehat{\beta} \Vert x\Vert_{Z}\quad\forall x\in X.
 \end{equation}
\end{lem}

\bibliographystyle{abbrv}
\bibliography{ester_fv.bib}

\end{document}